\documentclass[11pt]{amsart}
\usepackage{}
\usepackage{amssymb}

\usepackage{amsmath}
\usepackage{txfonts}

\input xy
\xyoption{all}

\usepackage{amsfonts}
\usepackage{mathrsfs}

\usepackage{latexsym}
\usepackage{graphicx}
\usepackage{amscd,amssymb,amsmath,amsbsy,amsthm}
\usepackage[all]{xy}
\usepackage[colorlinks,plainpages,backref,urlcolor=blue]{hyperref}
\usepackage{verbatim}

\usepackage{tikz}
\usetikzlibrary{arrows,calc}
\tikzset{
>=stealth',
help lines/.style={dashed, thick},
axis/.style={<->},
important line/.style={thick},
connection/.style={thick, dotted},
}

\newcommand{\nc}{\newcommand}
\nc{\rnc}{\renewcommand}
\nc{\bb}[1]{{\mathbb #1}}
\nc{\bbA}{\bb{A}}\nc{\bbB}{\bb{B}}\nc{\bbC}{\bb{C}}\nc{\bbD}{\bb{D}}
\nc{\bbE}{\bb{E}}\nc{\bbF}{\bb{F}}\nc{\bbG}{\bb{G}}\nc{\bbH}{\bb{H}}
\nc{\bbI}{\bb{I}}\nc{\bbJ}{\bb{J}}\nc{\bbK}{\bb{K}}\nc{\bbL}{\bb{L}}
\nc{\bbM}{\bb{M}}\nc{\bbN}{\bb{N}}\nc{\bbO}{\bb{O}}\nc{\bbP}{\bb{P}}
\nc{\bbQ}{\bb{Q}}\nc{\bbR}{\bb{R}}\nc{\bbS}{\bb{S}}\nc{\bbT}{\bb{T}}
\nc{\bbU}{\bb{U}}\nc{\bbV}{\bb{V}}\nc{\bbW}{\bb{W}}\nc{\bbX}{\bb{X}}
\nc{\bbY}{\bb{Y}}\nc{\bbZ}{\bb{Z}}
\nc{\mbf}[1]{{\mathbf #1}}
\nc{\bfA}{\mbf{A}}\nc{\bfB}{\mbf{B}}\nc{\bfC}{\mbf{C}}\nc{\bfD}{\mbf{D}}
\nc{\bfE}{\mbf{E}}\nc{\bfF}{\mbf{F}}\nc{\bfG}{\mbf{G}}\nc{\bfH}{\mbf{H}}
\nc{\bfI}{\mbf{I}}\nc{\bfJ}{\mbf{J}}\nc{\bfK}{\mbf{K}}\nc{\bfL}{\mbf{L}}
\nc{\bfM}{\mbf{M}}\nc{\bfN}{\mbf{N}}\nc{\bfO}{\mbf{O}}\nc{\bfP}{\mbf{P}}
\nc{\bfQ}{\mbf{Q}}\nc{\bfR}{\mbf{R}}\nc{\bfS}{\mbf{S}}\nc{\bfT}{\mbf{T}}
\nc{\bfU}{\mbf{U}}\nc{\bfV}{\mbf{V}}\nc{\bfW}{\mbf{W}}\nc{\bfX}{\mbf{X}}
\nc{\bfY}{\mbf{Y}}\nc{\bfZ}{\mbf{Z}}
\nc{\bfa}{\mbf{a}}\nc{\bfb}{\mbf{b}}\nc{\bfc}{\mbf{c}}\nc{\bfd}{\mbf{d}}
\nc{\bfe}{\mbf{e}}\nc{\bff}{\mbf{f}}\nc{\bfg}{\mbf{g}}\nc{\bfh}{\mbf{h}}
\nc{\bfi}{\mbf{i}}\nc{\bfj}{\mbf{j}}\nc{\bfk}{\mbf{k}}\nc{\bfl}{\mbf{l}}
\nc{\bfm}{\mbf{m}}\nc{\bfn}{\mbf{n}}\nc{\bfo}{\mbf{o}}\nc{\bfp}{\mbf{p}}
\nc{\bfq}{\mbf{q}}\nc{\bfr}{\mbf{r}}\nc{\bfs}{\mbf{s}}\nc{\bft}{\mbf{t}}
\nc{\bfu}{\mbf{u}}\nc{\bfv}{\mbf{v}}\nc{\bfw}{\mbf{w}}\nc{\bfx}{\mbf{x}}
\nc{\bfy}{\mbf{y}}\nc{\bfz}{\mbf{z}}

\nc{\mcal}[1]{{\mathcal #1}}
\nc{\calA}{\mcal{A}}\nc{\calB}{\mcal{B}}\nc{\calC}{\mcal{C}}\nc{\calD}{\mcal{D}}
\nc{\calE}{\mcal{E}} \nc{\calF}{\mcal{F}}\nc{\calG}{\mcal{G}}\nc{\calH}{\mcal{H}}
\nc{\calI}{\mcal{I}}\nc{\calJ}{\mcal{J}}\nc{\calK}{\mcal{K}}\nc{\calL}{\mcal{L}}
\nc{\calM}{\mcal{M}}\nc{\calN}{\mcal{N}}\nc{\calO}{\mcal{O}}\nc{\calP}{\mcal{P}}
\nc{\calQ}{\mcal{Q}}\nc{\calR}{\mcal{R}}\nc{\calS}{\mcal{S}}\nc{\calT}{\mcal{T}}
\nc{\calU}{\mcal{U}}\nc{\calV}{\mcal{V}}\nc{\calW}{\mcal{W}}\nc{\calX}{\mcal{X}}
\nc{\calY}{\mcal{Y}}\nc{\calZ}{\mcal{Z}}
\nc{\fA}{\frak{A}}\nc{\fB}{\frak{B}}\nc{\fC}{\frak{C}} \nc{\fD}{\frak{D}}
\nc{\fE}{\frak{E}}\nc{\fF}{\frak{F}}\nc{\fG}{\frak{G}}\nc{\fH}{\frak{H}}
\nc{\fI}{\frak{I}}\nc{\fJ}{\frak{J}}\nc{\fK}{\frak{K}}\nc{\fL}{\frak{L}}
\nc{\fM}{\frak{M}}\nc{\fN}{\frak{N}}\nc{\fO}{\frak{O}}\nc{\fP}{\frak{P}}
\nc{\fQ}{\frak{Q}}\nc{\fR}{\frak{R}}\nc{\fS}{\frak{S}}\nc{\fT}{\frak{T}}
\nc{\fU}{\frak{U}}\nc{\fV}{\frak{V}}\nc{\fW}{\frak{W}}\nc{\fX}{\frak{X}}
\nc{\fY}{\frak{Y}}\nc{\fZ}{\frak{Z}}
\nc{\fa}{\frak{a}}\nc{\fb}{\frak{b}}\nc{\fc}{\frak{c}} \nc{\fd}{\frak{d}}
\nc{\fe}{\frak{e}}\nc{\fFf}{\frak{f}}\nc{\fg}{\frak{g}}\nc{\fh}{\frak{h}}
\nc{\fri}{\frak{i}}\nc{\fj}{\frak{j}}\nc{\fk}{\frak{k}}\nc{\fl}{\frak{l}}
\nc{\fm}{\frak{m}}\nc{\fn}{\frak{n}}\nc{\fo}{\frak{o}}\nc{\fp}{\frak{p}}
\nc{\fq}{\frak{q}}\nc{\fr}{\frak{r}}\nc{\fs}{\frak{s}}\nc{\ft}{\frak{t}}
\nc{\fu}{\frak{u}}\nc{\fv}{\frak{v}}\nc{\fw}{\frak{w}}\nc{\fx}{\frak{x}}
\nc{\fy}{\frak{y}}\nc{\fz}{\frak{z}}

\newtheorem{theorem}{Theorem}[section]
\newtheorem{lemma}[theorem]{Lemma}
\newtheorem{corollary}[theorem]{Corollary}
\newtheorem{prop}[theorem]{Proposition}

\theoremstyle{definition}
\newtheorem{definition}[theorem]{Definition}
\newtheorem{example}[theorem]{Example}
\newtheorem{remark}[theorem]{Remark}
\newtheorem{question}[theorem]{Question}
\newtheorem{problem}[theorem]{Problem}

\newtheorem{thm}{Theorem}

\DeclareMathOperator{\rank}{rank} \DeclareMathOperator{\gr}{gr}
\DeclareMathOperator{\im}{im} \DeclareMathOperator{\coker}{coker}
\DeclareMathOperator{\codim}{codim} 
 \DeclareMathOperator{\Sym}{Sym}
\DeclareMathOperator{\ch}{ch} \DeclareMathOperator{\supp}{supp}
 
\DeclareMathOperator{\Hom}{{Hom}} 
\DeclareMathOperator{\Ext}{{Ext}}
\DeclareMathOperator{\sHom}{{\mathscr{H}om}}

\DeclareMathOperator{\Hilb}{{Hilb}}

\DeclareMathOperator{\Spec}{{Spec}} 

 \DeclareMathOperator{\End}{End}

\DeclareMathOperator{\Coh}{Coh}
\DeclareMathOperator{\Mod}{Mod\hbox{-}}
\DeclareMathOperator{\Gm}{\bbG_m}

\DeclareMathOperator{\SL}{SL}

\DeclareMathOperator{\Bl}{Bl}

\DeclareMathOperator{\Sp}{Sp}

\DeclareMathOperator{\Qcoh}{Qcoh}
\DeclareMathOperator{\red}{red}
\DeclareMathOperator{\Perv}{Perv}

\newcommand{\sO}{\mathcal{O}}

\newcommand{\sA}{\mathscr{A}}

\newcommand{\hh}{\mathfrak{h}}
\newcommand{\sV}{\mathscr{V}}
\newcommand{\sE}{\mathscr{E}}
\newcommand{\sL}{\mathscr{L}}
\DeclareMathOperator{\Stab}{Stab}

\newcommand{\CC}{\mathbb{C}}
\newcommand{\catA}{\mathfrak{A}}

\DeclareMathOperator{\Res}{Res}
\DeclareMathOperator{\Ind}{Ind}

\newcommand{\surj}{\twoheadrightarrow}
\newcommand{\inj}{\hookrightarrow}

\def\angl#1{{\langle #1\rangle}}

\newcommand{\Aff}{\bbA}
\newcommand{\frob}{^{(1)}}

\newcommand{\PP}{\bbP}
\newcommand{\ZZ}{\bbZ}
\newcommand{\QQ}{\bbQ}

\newcommand{\C}{\bbC}
\newcommand{\RR}{\bbR}

\newcount\cols
{\catcode`,=\active\catcode`|=\active
 \gdef\Young(#1){\hbox{$\vcenter
 {\mathcode`,="8000\mathcode`|="8000
  \def,{\global\advance\cols by 1 &}%
  \def|{\cr
        \multispan{\the\cols}\hrulefill\cr
        &\global\cols=2 }%
  \offinterlineskip\everycr{}\tabskip=0pt
  \dimen0=\ht\strutbox \advance\dimen0 by \dp\strutbox
  \halign
   {\vrule height \ht\strutbox depth \dp\strutbox##
    &&\hbox to \dimen0{\hss$##$\hss}\vrule\cr
    \noalign{\hrule}&\global\cols=2 #1\crcr
    \multispan{\the\cols}\hrulefill\cr%
   }
 }$}}
}

\title[Stability conditions for symplectic resolutions]
{Real variations of stability conditions for noncommutative symplectic resolutions}
\date{\today}

\author[G.~Zhao]{Gufang~Zhao}
\address{Department of Mathematics,
Northeastern University, Boston, MA 02115, USA}
\curraddr{Max-Planck-Institute,
Vivatgasse 7, 53111 Bonn, Germany}
\email{gufangzhao@zju.edu.cn}


\begin{document}
\maketitle
\begin{abstract}
A localization theorem for the cyclotomic rational Cherednik algebra $H_c=H_c((\mathbb{Z}/l)^n\rtimes \mathfrak{S}_n)$ over a field of positive characteristic has been proved by Bezrukavnikov, Finkelberg and Ginzburg. Localizations with different parameters give different $t$-structures on the derived category of coherent sheaves on the Hilbert scheme of points on a surface. In this short note, we concentrate on the comparison between different $t$-structures coming from different localizations. When $n=2$, we show an explicit construction of tilting bundles that generates these $t$-structures. These $t$-structures are controlled by a real variation of stability conditions, a notion related to Bridgeland stability conditions. We also show its relation to the topology of Hilbert schemes and irreducible representations of $H_c$.
\end{abstract}

\tableofcontents

\section{Introduction}
For a finite dimensional vector space $V$, equipped with a symplectic form, and a finite subgroup of the symplectic group $\Gamma\subseteq \Sp(V)$, the quotient $V/\Gamma$ is a Poisson variety, and the bracket is non-degenerate on the smooth part. Suppose that we have a resolution of singularity $\pi: X\to V/\Gamma$, with a symplectic form on $X$ which coincide with that on $V/\Gamma$ when restricted to the smooth locus. Such resolutions are called \textsl{symplectic resolutions}.

Simplest interesting examples of symplectic resolutions are the minimal resolutions of Kleinian singularities. More precisely, for a finite subgroup $\Gamma\subseteq\Sp(\Aff^2)$, the quotient $\Aff^2/\Gamma$ has a unique symplectic resolution, denoted by $\widetilde{\Aff^2/\Gamma}$. More generally, the symmetric product $\Sym^n(\Aff^2/\Gamma)$ has a symplectic resolution given by the Hilbert scheme of points $\Hilb^n(\widetilde{\Aff^2/\Gamma})$.

Bezrukavnikov and Kaledin proved in \cite{BK04}, that for any symplectic resolution $\pi: X\to V/\Gamma$, there exists a vector bundle $\calV$ on $X$, such that $\End_{\calO_X}(\calV)\cong\calO_V\#\Gamma$, and $R\Hom_{\calO_X}(\calV,\bullet)$ induces an equivalence of triangulated categories $D^b(\Coh(X))\cong D^b(\Mod\End(\calV))$. In the terminology of \cite{BO}, the noncommutative algebra $\End(\calV)$, viewed as a coherent sheaf on $V/\Gamma$, is a noncommutative resolution of singularity, which is clearly a noncommutative crepant resolutionn in the sense of \cite{vdB}. As a consequence of the Bezrukavnikov and Kaledin theorem, all symplectic resolutions of $V/\Gamma$ are derived equivalent to each other.
When $V=\Aff^2$, the theorem of Bezrukavnikov and Kaledin specializes to the classical derived McKay correspondence about Kleinian singularities. When $V=\Aff^{2n}\cong(\Aff^2)^n$ and $\Gamma=\fS_n$ acting by permuting the $\Aff^2$-factors, and $X=\Hilb^n(\Aff^2)$, the bundle $\calV$ constructed by Bezrukavnikov and Kaledin is related to the Procesi bundles studied by Haiman in \cite{H02}.
It is worth mentioning that the construction of the noncommutative resolutions given in \cite{BK04} comes from quantization of symplectic varieties over fields of positive characteristic. 

An example, generalizing the case when $V=\Aff^{2n}$ and $\Gamma=\fS_n$, is the following.
We work over a separably closed field $k$  of characteristic $p>>0$. Let $\Gamma_n:=(\ZZ_r)^n\rtimes\mathfrak{S}_n$ acting on $\hh=\Aff^{n}$ in the natural way, i.e., the $i$-th factor of $\ZZ_r$ acts on the $i$-th factor of $\Aff^1$, and $\fS_n$ permutes the coordinates. Let $V=\hh\oplus\hh^*\cong\Aff^{2n}$ be endowed with the diagonal action of $\Gamma_n$. The action preserves the natural symplectic form on $V$. A symplectic resolution of $\Aff^{2n}/\Gamma_n$ is given by $\Hilb^n(\widetilde{\Aff^2/\ZZ_r})$ where $\widetilde{\Aff^2/\ZZ_r}$ is the minimal resolution of $\Aff^2/\ZZ_r$. Let $\calW(\hh)$ be the Weyl algebra. Let $V\frob$ be the Frobenius twist of $V$. Then $\calW(\hh)$ is a coherent sheaf of algebras, which is an Azumaya algebra. Hence, $\calW(\hh)\#{\Gamma_n}$ is also a coherent sheaf of algebras. One can easily convince himself that $\calW(\hh)\#{\Gamma_n}$ has finite global dimension, therefore is a \textit{noncommutative desingularization} of $\Aff^{2n}/\Gamma_n$. So is the algebra $\calW(\hh)^{\Gamma_n}$, which is Morita equivalent to $\calW(\hh)\#\Gamma_n$. When restricted to the formal neighborhood of $0$ in $V^{(1)}$, the algebra $\calW(\hh)\#\Gamma_n$ in turn is Morita equivalent to $k[V\frob]\#\Gamma_n$. 

The  {\it cyclotomic rational Cherednik algebra} $H_c$ is a deformation of $\calW(\hh)\#{\Gamma_n}$. The parameter space of the deformation is a vector space spanned by the conjugacy classes of reflections in $\Gamma_n$, which is naturally isomorphic to $H^2(X_{\Gamma_n})$.  The precise definition of the cyclotomic rational Cherednik algebra is recalled in Section~\ref{sec: quant_char_p}. 
The (non-unital) subalgebra $^sH_c:=eH_ce\subset H_c$ is called {\it the spherical Cherednik algebra}, where $e:=\sum_{\gamma\in \Gamma_n} \gamma$. If  $^sH_c$ is Morita equivalent to $H_c$, then the value $c$ is called \textit{spherical value}. Otherwise we say $c$ is \textit{aspherical}.  The aspherical values form an affine hyperplane arrangement in the space of parameters. For any value of $c$, the algebra $H_c$ always has finite global dimension. However, the spherical subalgebra $^sH_c$ has finite global dimension if and only if it is Morita equivalent to $H_c$. Similar to $\calW(\hh)^{\Gamma_n}$, the algebra $^sH_c$ has a big Frobenius center $k[\Aff^{2n(1)}]^{\Gamma_n}$. In other word, $^sH_c$ is a coherent sheaf of algebras on $V^{2n(1)}/\Gamma_n$. For any central character $\chi$ (i.e., a closed point in $V^{2n(1)}/\Gamma_n$), let $\Mod_\chi \null^sH_c$ be the category of modules over $^sH_c$ which are set-theoretically supported on the closed point $\chi$. The irreducible objects in the category $\Mod_\chi ^sH_c$ are naturally indexed by $\hbox{Irrep}(\Gamma_n)$.

Let $\Hilb^{(1)}$ be the Frobenius twist of $\Hilb:=\Hilb^n(\widetilde{\Aff^2/\ZZ_r})$.
Let $\Coh_0\Hilb\frob$ be the category of coherent sheaves on $\Hilb\frob$ which are set-theoretically supported on the zero-fiber of the Hilbert-Chow morphism.
It is shown by Bezrukavnikov-Finkelberg-Ginzburg that there is a tilting bundle $\sE_c$ on $\Hilb^{(1)}$, such that $\End(\sE_c)|_{\hat{0}}\cong (^sH_c)|_{\hat{0}}$, where $\hat{0}$ is the normal neighborhood of $0$ in $\Aff^{2n(1)}/\Gamma_n$.
In particular, for spherical values $c$, the algebra $^sH_c$ has finite global dimension. Consequently,
there is a derived equivalence $$D^b(\Coh_0\Hilb^{(1)})\cong D^b(\Mod_0\null^sH_c).$$


Assume the characteristic of the base field $k$ is $p\gg0$. Then for any spherical value of the parameter $c\in H^2(X_{\Gamma_n};\bbQ)$
, the derived equivalence given by \cite{BEG} endows $D^b(\Coh(X_{\Gamma_n}))$ with a $t$-structure, whose heart is the image of $\Mod_0\null^sH_c$ under this equivalence.
The aspherical values form an affine hyperplane arrangement in the space of parameters, which divide the parameter space into facets. An open facet is called an alcove, and a codimension-1 facet is called a wall.
If the parameter $c\in H^2(X_{\Gamma_n};\bbQ)$ varies inside a single alcove, the $t$-structure stays constant. The dimension of the irreducible object $L_c(\tau;p)$ considered as a $k$-vector space, varies polynomially with respect to $c$. This polynomial will be referred to as the dimension polynomial.





The dimension polynomials are related to the topology of the Hilbert schemes. It is shown in Proposition~\ref{prop: chern_dim} that the solution to the Chern character problem determines the dimension polynomials $\dim L_c(\tau,p)$ of the irreducible objects.

It is well-known (see \cite{Kuz01}) that a symplectic resolution of $\Aff^{2n}/\Gamma_n$ can be constructed as a Nakajima quiver variety associated to the extended Dynkin quiver.
For a suitable choice of the stability condition, the Nakajima variety is isomorphic to $\Hilb^n(\widetilde{\Aff^2/\Gamma_1})$, where $\widetilde{\Aff^2/\Gamma_1}$ is the minimal resolution of the Kleinian singularity $\Aff^2/\Gamma_1$.
As an intermediate step of studying the stability conditions, in the example when $\Gamma_1=\bbZ/r\bbZ$ and $n=2$, using this quiver description, the Chern character map has been written down explicitly in Proposition~\ref{prop: chern char}. In general the calculation of the Chern character map is difficult. But it is easier, at least in some cases, to calculate the dimension polynomials.

For an integral parameter $m$, let $Q_m$ be the $m$-quasi-invariants in $k[\mathfrak{h}]$. As $\Gamma_n$-$^sH_m$ bimodule,
$ Q_m=\oplus_{\tau\in\hbox{Irrep}(\Gamma_n)}\tau^*\otimes M_m(\tau).$ Let $\widetilde{Q_m}$ be the quasi-invariants on the Frobenius neighborhood of 0. A resolution of $\widetilde{Q_m}$:
$\cdots\to Q_m\otimes \wedge^2\mathfrak{h}^{(1)}\to Q_m\otimes \mathfrak{h}^{(1)}\to Q_m.$

\begin{thm}
Fix a character $i$ of $\ZZ_r$. Let  $\tau(i)$ be the 1-dimensional representation of $\Gamma_n=(\ZZ_r)^n\rtimes\fS_n$ on which $\ZZ_r$ acts by the character $i$ and $\fS_n$ acts by the sign representation.

The Poincar\'e series of $L_m(\tau(i))$ is
\[{\frac{t^{ni}\prod_{k=0}^{n-1}(1-t^{rk+m_0n+p+1+rm_{i+1}})}{\prod_{k=1}^n(1-t^{kr})}}.\]
\end{thm}

Using the induction and restriction functors, this theorem gives an algorithm to calculate the dimension polynomials of the irreducible objects as long as the parameter $m$ is in the foundamental alcove (the alcove containing 0). But away from the foundamental alcove, the combinatorics becomes complicated and we can only deal with the case when $n=2$ in the current paper.

We  define \[Z_\tau(x)=\lim_{p\to\infty}p^{-n}\dim_kL_{cp}(\tau;p).\]
We consider the collection of polynomials $\{Z_\tau(x)\mid\tau\in\hbox{ Irrep}(\Gamma_n)\}$ as a polynomial map $H^2(X_{\Gamma_n};\bbR)\to K_0(X_{\Gamma_n})^*\otimes\bbR$. Let $\phi$ be the assignment associating to each alcove the $t$-structure on $D^b(\Coh(X_{\Gamma_n}))$ coming from $\Mod \calH_c(\Gamma_n)$ for $c$ lying in this alcove. In general, it is conjectured by Bezrukavnikov and Okounkov that the pair $(\phi,Z)$ is a real variation of stability conditions in the sense of \cite{ABM}.
The notion of real variation of stability conditions, as well as the more precise meaning of this conjecture, will be discussed in detail in Section~\ref{sec: quant_char_p}. In the Introduction we only make precise what has been achieved in the current paper when $n=2$ as Theorem~\ref{thm: main_intr}.

Assume $ n=2$.
Let $\phi:\{\hbox{alcoves}\}\to \{t\hbox{-structures}\}$ be the map assigning each alcove $A$ the $t$-structure on $D^b(\Coh_0\Hilb)$ whose heart is $Mod_0\null ^sH_c(\Gamma_2)\subseteq D^b(\Coh_0(\Hilb))$ for $c\in A$. Let the central charge polynomials $Z_\tau(\nu)$ for $\tau\in\hbox{Irrep}{\Gamma_2}$ be defined as above.
\begin{thm}\label{thm: main_intr}
The pair $(\phi,Z)$ is a real variation of stability conditions.

More concretely, for any alcove $A$, let $\catA:={\hbox{heart of }}\phi(A)$. We have,
\begin{enumerate}
  \item for any $x\in A$, we have $ Z_L(x)> 0$ for any simple object $L\in \catA$;
  \item for any $A'$, sharing  with $A$ a codimension-1 wall $H$.

Let $\catA_i$ be the Serre subcategory of $\catA$ generated by the simple objects $L\in \catA$ with corresponding $ Z_L(x)$ vanishing of
order $\geq i$ on $H\rangle$. Then,
\begin{itemize}
\item the $T(A')$ is compatible with the filtration on $T(A)$;
\item on $\gr_i(\catA) = \catA_i/\catA_{i+1}$, $\phi(A')$ differs by $[i]$ from  $\phi(A)$.
\end{itemize}
\end{enumerate}
\end{thm}

An explicit description of the derived equivalences for any two adjacent alcoves in this case can be found in Section~\ref{sec: t-struct_alco}.  Where are only two types: $\PP^2$-semi-reflection, and tilting with respect to suitable torsion theory. The question how the tilting generators change under $\PP^n$-semi-reflection in general is studied in \S~\ref{sec: mut}, which is interesting on its own rights.

There are two prototypical examples of $\PP^n$-semi-reflections.
\begin{example}
Let $\Perv(\PP^n)$ be the category of perverse constructible sheaves with respect to the usual stratification of $\PP^n$. Similarly we have the category $\Perv((\PP^n)^\vee)$ on the dual projective space $(\PP^n)^\vee$. Let $R:D^b(\Perv(\PP^n))\to D^b(\Perv((\PP^n)^\vee))$ be the Radon transform with kernel the incidence locus. Then $R(\Perv(\PP^n))$ is the semi-reflection of $\Perv((\PP^n)^\vee)$ with respect to the $\PP^n$-object $\CC_{\PP^n}[n]$.
\end{example}
\begin{example}
Let $D^b(\Coh_0 T^*\PP^n)$ be the derived category of coherent sheaves on $T^*\PP^n$ set-theoretically supported on the zero-section, and let $\catA$ be the heart of the $t$-structure whose projective generator is the Beilinson's tilting bundle $\oplus_{i=0}^n\calO(-i)$. Similarly let $\catA'$ be the the heart of the $t$-structure in $D^b(\Coh_0 T^*(\PP^n)^\vee)$ whose projective generator is given by the Beilinson's tilting bundle on $(\PP^n)^\vee$. Let $FM:D^b(\Coh_0 T^*\PP^n)\to D^b(\Coh T^*(\PP^n)^\vee)$ be the Fourier-Mukai transform constructed by Namikawa in \cite{Nam}. Then $FM(\catA)$ is the semi-reflection of $\catA'$ with respect to the $\PP^n$-object $\calO_{\PP^n}(-n)$. (See also \cite{TU}.)
\end{example}
The following results, which is a scene from Section~\ref{sec: mut}, tells us about the projective generator in the heart of the $t$-structure obtained from $\PP^n$-semi-reflection.

A more general set-up for the $\PP^n$-semi-reflection is the following. Let $\bfX$ be a smooth variety which is projective over $\Spec A$. We assume moreover that the map $\pi: \bfX\to \Spec A$ is $\Gm$-equivariant, such that this $\Gm$-action gives a  deformation retraction of $\bfX$ to $X=\pi^{-1}(\Spec A/m)$, the fiber over $A/m$.
Let $\{P_\alpha\mid\nabla\}$ be a collection of $\Gm$-equivariant tilting bundles on $\bfX$, and denote $\End(\oplus_{\alpha\in \nabla}P_\alpha)$ by $E$. Let $\catA$ be the category of finitely generated $E$-modules which are set-theoretically supported on $A/m$.

The following fact about $\PP^n$-semi-reflection is proved in Corollary~\ref{cor:perver-smaller_pervers}. (The result also holds if $\calA$ is a finite length abelian category with enough projectives, e.g., the category of perverse constructible sheaves.)
Assume $S_\theta$ is a simple object has vanishing $\Ext^1(S_\theta,S_\theta)$. We endow $\catA$ with the filtration that $0=\catA_0\subseteq\catA_1\subseteq\catA_2=\catA$ where $\catA_1=\langle S_\theta\rangle$. Assume for the perversity function $p$ with $p(1)=0$ and $p(2)=n$ we have a perverse equivalence $(t,t',p)$ such that the projective covers of the simple objects in the heart of $t'$ have representatives lying in $E\hbox{-}mod$. Then for any $p'$ with $p'(1)=0$ and $p'(1)\leq n$ the perverse equivalence $(t,t'',p')$ exists, and the projective covers of the simple objects in the heart of $t'$ have representatives lying in $E\hbox{-}mod$. Moreover, the projective generators of these $t$-structures are given by the \textit{truncated mutations} defined in Section~\ref{sec: mut}.
\subsection*{Acknowledgements}
This short note grows out of part of my PhD thesis in Northeastern University. I would like to give my special thanks to Prof. Roman Bezrukavnikov for introducing me to this field, suggesting this project to me, and had numerous times of discussions without which I could have gone nowhere. I also learned a lot about this subject from courses, seminars, and private communications with Ivan Losev and Ben Webster. 

\section{Truncated mutations}\label{sec: mut}
\subsection{Tilting with respect to a simple object}
Suppose $\catA \subset D$ is the heart of a bounded $t$-structure and is a finite length abelian category.
A torsion pair in $\catA$ is a pair of full subcategories
$(\mathcal{T},\mathcal{F})$ with the property that $\Hom(T, F) = 0$ for any $T \in \mathcal{T}$ and $F \in \mathcal{F}$, and  every object $E \in \catA$ fits into a short exact sequence
$$0\to T \to E \to F \to 0.$$
The following Lemma is due to Happel, Reiten, and Smal{\o}. (See also \cite{Bri}, Proposition 5.4.)
\begin{lemma}
Suppose $\catA\subset D$ is the heart of a bounded $t$-structure on a triangulated category $D$. Suppose $(\mathcal{T},\mathcal{F})$ is a torsion pair in $\catA$. Then the full subcategory $R_\tau\catA=\{E\in D\mid H^i(E)=0\hbox{ for }i\notin\{-1,0\}\hbox{, }H^{-1}\in\mathcal{F}\hbox{ and }H^0(E)\in\mathcal{T}\}$ is the heart of a bounded $t$-structure.
\end{lemma}
The new $t$-structure $R_\tau\catA$ is called the \textit{(right) tilting} of $\calA$ with respect to the torsion pair $(\mathcal{T},\mathcal{F})$.

The following Lemma gives a criterion for simple objects to be in the heart of the new $t$-structure $R_{\mathcal{T}}\catA$.
\begin{lemma}[See Lemma~2.4 in \cite{W10}]\label{lem_simple_R_til}
Let $\mathcal{T}$ be a torsion theory in the heart $\catA$ of a $t$-structure. Then
any simple object in $R_{\mathcal{T}}\catA$ lies either in $\mathcal{T}$  or in $\mathcal{F}[1]$ and
\begin{enumerate}
\item $T\in \mathcal{T}$ is simple in $R_{\mathcal{T}}\catA$ iff there are no exact triangles
\[T'\to T\to T''\to T'[1]\txt{ or  }\ T'\to T\to F'[1]\to T'[1]\]
with $T'$, $T''\in \mathcal{T}$ and $F'\in \mathcal{F}$ and all non-zero;
\item $F[1]\in \mathcal{F}[1]$ is simple in $R_{\mathcal{T}} \catA$ iff there are no exact triangles
\[F'\to F\to F''\to F'[1]\txt{ or  }\ T'[-1]\to F'\to F\to T'\]
with $F'$, $F''\in \mathcal{F}$ and $T'\in  \mathcal{T}$ and all non-zero.
\end{enumerate}
\end{lemma}

Given a simple object $S \in \catA$, define $\langle S\rangle\subset\mathfrak{A}$ to be the full subcategory consisting
of objects $E \in\catA$ all of whose simple factors are isomorphic to $S$. One can easily check that the pair $\mathcal{F}=\langle S\rangle$ and $\mathcal{T}=\{E\mid \Hom(E,S)=0\}$ is a torsion pair.

\par
The (right) tilted subcategory of $\catA$ with respect to $S$ is defined to be $$R_S\catA = \{E \in D \mid H^i(E)=0\hbox{ for }i\neq-1,0\hbox{, }H^{-1}(E)\in\langle S\rangle\hbox{ and }\Hom(H^0(E),S)=0 \}.$$ Similarly there is a notion of left tilting $L_S\catA$.
\par
In the heart $R_S\catA$ of the new $t$-structure, $S[1]$ is a simple object. We can consider the tilting of $R_S\catA$ with respect to $S[1]$. But for this to work we need the abelian category $R_S\catA$ to be of finite length. Now we give a sufficient condition to guarantee this property.

Fix a simple object $S_{\theta}$ in an abelian category $\catA$, for another simple object $S_\alpha$ we use $S_{\alpha}'$ to denote the universal extension of $S_{\theta}$ by $S_{\alpha}$, which is the middle term in the tautological short exact sequence
\begin{equation}\label{eq: universal-ext}
0\to S_{\theta}\otimes\Ext^1(S_\alpha,S_{\theta})^*\to S_\alpha'\to S_\alpha\to 0.
\end{equation}
Note that $ S_\alpha'$ has cohomology concentrated in degree zero.
\begin{lemma}\label{lem: simple}\footnote{The author is grateful to Sasha Kuznetzov for pointing out a better set-up to carry out iterated tiltings studied in his work in preparation.}
Let $\mathfrak{A}$ be an abelian category of finite length with the complete set of  pairwise distinct simple objects $\{S_\alpha\mid\alpha\in\nabla\}$ indexed by a finite indexing set $\nabla$. Suppose that $\Ext^1(S_{\theta},S_{\theta})=0$, then $R_{S_{\theta}}\mathfrak{A}$ is still a finite length category whose set of all simple objects is $\{S_{\theta}[1]\}\cup\{S_\alpha'\mid\alpha\neq\theta\}$.
\end{lemma}
\begin{proof}
First note that $S_\theta[1]$ and $S_\alpha'$ are simple objects in $R_{S_{\theta}}\mathfrak{A}$. For $S_\theta[1]$, this is clear by Lemma~\ref{lem_simple_R_til}. For $S_\alpha'$, applying $\Hom(-,S_\theta[1])$ to the short exact sequence (\ref{eq: universal-ext}), we get \[\cdots\to\Hom(S'_\alpha,S_\theta)\to\Hom(S_\theta,S_\theta)\otimes\Ext^1(S_\alpha,S_\theta)\surj\Ext^1(S_\alpha,S_\theta)\to\Ext^1(S'_\alpha,S_\theta)\to 0.\] This shows $\Hom(S_\alpha',S_\theta[1])=0$. The composition factors of $S'_\alpha$ are $S_\alpha$ and some copies of $S_\theta$, therefore, there is no exact triangle $T'\to S_\alpha'\to T''\to T'[1]$ with $T'$ and $T''\in \mathcal{T}$. So, Lemma~\ref{lem_simple_R_til} yields the simplicity of $S'_\alpha$.

We only need to show that any object $E$ in $R_{S_{\theta}}\mathfrak{A}$ has a finite filtration with sub-quotients isomorphic to $S_\theta[1]$ and $S_\alpha'$.
We use induction on the total number of copies of $S_\alpha$  for $\alpha\neq \theta$ occurring as composition factors of $H^0(E)$. If the number is zero, this means the cohomology of $E$ is concentrated in degree -1, and therefore, is a direct sum of $S_\theta[1]$. Otherwise, via taking cokernel of maps from $S_\theta[1]$ in the abelian category $R_{S_{\theta}}\mathfrak{A}$, we can assume the cohomology of $E$ is concentrated in degree zero. There is some $\alpha\neq \theta$ such that $\Hom(H^0(E),S_\alpha)\neq0$ which implies $\Hom(E,S'_\alpha)\neq0$.
As $S'_\alpha$ is simple in $R_{S_\theta}\catA$, this map must be surjective. Let the kernel be $K$. Taking cohomology long exact sequence with respect to the original $t$-structure of the exact triangle \[K\to E\to S'_\alpha\to K[1],\]  we know that in the composition factors of $H^0(K)$ the total number of copies of $S_\alpha$  with $\alpha\neq \theta$ has been reduced by 1.
\end{proof}
We will denote $S_{\theta}[1]$ by $S_{\theta}'$.

\par
Now we assume $\Ext^1(S_{\theta},S_{\theta})$ vanish.  For a fixed $\theta\in\nabla$, let $S^0_\alpha=S_\alpha$. Recursively we define $S_\alpha^i$ to be the universal extension of $S_{\theta}^{i-1}$ by $S_\alpha^{i-1}$ for $\alpha\neq\theta$, and $S_{\theta}^i=S_{\theta}^{i-1}[1]$.
\par
Since $\Ext^i(S^i_\theta,S^i_\theta)=\Ext^i(S_\theta,S_\theta)$, we have the following proposition.

\begin{prop}\label{prop: inter_tilt}
Let $\mathfrak{A}$ be an abelian category of finite length with simple objects $\{S_\alpha\mid\alpha\in\nabla\}$ for a finite set $\nabla$. Suppose that $\Ext^1(S_{\theta},S_{\theta})=0$, then $R_{S_{\theta}[i-1]}R_{S_{\theta}[i-2]}\cdots R_{S_{\theta}}(\mathfrak{A})$ is still a finite length category whose simple objects are $\{S_\alpha^i\mid\alpha\}$.
\end{prop}
\par
Let $\catA\subset D$ be the heart of some $t$-structure of $D$, following Bridgeland, we denote the region in the stability space corresponding to $\catA$ by $U(\catA)$.
Suppose $(Z,\catA)$ is a stability condition in the boundary of the region $U(\catA)$.
Then
there is some $i$ such that $Z(S_i)$ lies on the real axis. Assume that $\im Z(S_j) > 0$ for every $j\neq i$. Since each object $S_i$ is
stable for all stability conditions in $ U(\catA)$, each $S_i$ is at least semistable in
$(Z,\catA)$, and hence $Z(S_i)$ is nonzero.
\begin{lemma}[\cite{Bri}, Lemma~5.2]
Suppose the heat $\catA\subset D$ of a bounded $t$-structure has finite length and $n$ simple objects, then $U(\catA)$ is isomorphic to $\mathbb{H}^n$ where $\mathbb{H}$ is the upper half plane in $\CC$ together with the positive real-axis.
\end{lemma}
\par
For a stability condition $(\catA,Z)$ on a wall of codimension 1, then $Z(S)$ takes positive real values on that wall for some simple object $S$. If $R_S\catA$ has the same finiteness property, then $U(\catA)$ and $R_S\catA$ glues together along this wall.

\begin{corollary}
If $S$ is a simple object in $\catA$ without self-extension, then $\Stab(\catA)$ has a locally closed subspace obtained by gluing $\mathbb{H}^n$'s together along the copy of $\mathbb{H}$ corresponding to the simple object $S$.
\end{corollary}

\subsection{Perverse equivalences}

Our main reference for this subsection is \cite{CR}.

For a Serre subcategory $\mathcal{I}$ of an exact category $\catA$, the thick subcategory in $D^b(\catA)$ generated by $\mathcal{I}$ will be denoted by $\langle\mathcal{I}\rangle$.

Let $\catA$ and $\catA'$ be two exact categories endowed with filtrations $0 = \catA_0\subseteq\catA_1\cdots\subseteq\catA_r = \catA $ and
$0 = \catA'_0\subseteq\catA'_1\cdots\subseteq\catA'_r = \catA' $ by Serre subcategories. Let $p : \{0,\dots,r\}\to \ZZ$ be any function. The notion of perverse equivalence with respect to this filtration and perversity function $p$ is defined in \cite{CR}.

\begin{definition}
An equivalence $F : D^b(\catA)\to D^b(\catA')$ is \textit{perverse} relative to the filtrations $(\catA_\bullet,\catA'_\bullet)$ and function $p$, if for any $i$, the functor 
$F$ restricts to equivalences $\langle\catA_i\rangle\cong\langle\catA'_i\rangle$, and there is an equivalence $\catA_i={\catA_i}/{\catA_{i-1}}\to \catA'_i={\catA'_i}/{\catA'_{i-1}}$ compatible with the following equivalence induced by $F$:
\[F[p(i)]:{\langle\catA_i\rangle}/{\langle\catA_i\rangle}\cong {\langle\catA'_i\rangle}/{\langle\catA'_i\rangle}.\]
\end{definition}

In the case when $\catA'$ is not endowed with filtration, we make the following convention.  We define the filtration on $\catA'$ by $\catA'_i=\catA'\cap F(\catA_i)$, and we talk about perverse equivalence only in the case when each $\catA'_i$ defined this way is a Serre subcategory of $\catA'$.

There is also a notion of perverse data when talking about two $t$-structures $t$ and $t'$ on the same triangulated category with a filtration $\mathfrak{T}_*$ with respect to a perversity function $p$ defined in \cite{CR}. We say the quadruple $(t,t',\mathfrak{T}_*,p)$ is a perverse data if both $t$ and $t'$ are compatible with the filtration $\mathfrak{T}_*$, and for each $i$ we have $t|_{\mathfrak{T}_i/\mathfrak{T}_{i-1}}=t'|_{\mathfrak{T}_i/\mathfrak{T}_{i-1}}[p(i)]$

The followings are some basic properties of perverse equivalences.
\begin{prop}[See \cite{CR}]\label{prop: Chuang-Rouquier}
Notations as above, we have the following.
\begin{enumerate}
\item If $F$ is a perverse equivalence relative to $(\catA_\bullet,\catA'_\bullet,p)$, then $F^{-1}$ is perverse relative to $(\catA'_\bullet,\catA_\bullet,-p)$.
\item In this case, let $\catA''$ be another exact category endowed with filtration $\catA''_\bullet$ by Serre subcategories, and let $p':\{0,\dots,r\}\to\ZZ$ be another map. Assume $F':D^b(\catA')\to D^b(\catA'')$ is a perverse equivalence relative to $(\catA'_\bullet,\catA''_\bullet,p)$. Then $F'\circ F$ is a perverse equivalence relative to $(\catA_\bullet,\catA''_\bullet,p+p')$.
\item If we have two perverse data $(t,t',\mathfrak{T}_*,p_1)$ and $(t,t'',\mathfrak{T}_*,p_2)$ with $p_1=p_2$, then $t'=t''$.
\end{enumerate}
\end{prop}

\subsection{Truncated mutations}\label{subsec: tilting}
Let $E$ be an associative algebra over a base field $k$. Let $\{P_\alpha\mid\alpha\in\nabla\}$ be the set of (isomorphism classes of) indecomposable projective objects in the category $E\hbox{-}mod$. We assume $\nabla$ to be a finite set. Then it is well-know that $E$ is Morita equivalent to $\End(\oplus P_\alpha)$.
Let $\catA\inj E\hbox{-}mod$ be a fully-faithful exact embedding of a finite length abelian subcategory with finite dimensional $\Hom$'s, which preserves $\Ext$'s. Assume the (isomorphism classes of) simple objects $\{S_\alpha\mid\alpha\in\nabla\}$ in $\catA$ are indexed by the same set $\nabla$, such that each $S_\alpha$ is simple in $E\hbox{-}mod$ and its projective cover is $P_\alpha$. We make one additional assumption:
For each pair $(\theta,\alpha)$ in $\nabla$, let $S_{\alpha,\theta}$ be the universal extension, fitting into the short exact sequence $$0\to \Ext^1(S_\alpha,S_\theta)^*\otimes S_\theta\to S_{\alpha,\theta}\to S_\alpha\to 0.$$ We assume the map $\Hom(P_\theta,P_\alpha)\to\Ext^1(S_\alpha,S_\theta)^*$, induced by the composition morphism $\Hom(P_\alpha,S_{\alpha,\theta})\otimes\Hom(P_\theta,P_\alpha)\to\Hom(P_\theta,S_{\alpha,\theta})$, is surjective.

In the case when $E$ is finite dimensional over $k$, the only example of such subcategory $\catA$ is $E\hbox{-}mod$ itself. A non-trivial example of such subcategory will be given in Subsection~\ref{subsec: truncate_geom}.

Fix an $\theta \in \nabla$.  For each $\alpha\neq\theta $, we fix a section of the surjection $\Hom(P_\theta,P_\alpha)\to\Ext^1(S_\alpha,S_\theta)^*$, and denote the image of the section by $\Hom(P_{\theta},P_\alpha)_{t_\alpha}$. We define $P'_\alpha$ to be $P_\alpha$ if $\alpha\neq\theta $, and $P'_{\theta}$ to be the mapping cone in $D^b(E\hbox{-}mod)$ of the natural map $P_{\theta}\to \oplus_{\alpha\neq\theta }P_\alpha\otimes\Hom(P_{\theta},P_\alpha)_{t_\alpha}^*$.
\begin{definition}
If the natural map $P_{\theta}\to \oplus_{\alpha\neq\theta }P_\alpha\otimes\Hom(P_{\theta},P_\alpha)_{t_\alpha}^*$ is injective,
the set $\{P'_\alpha\mid\alpha\in\nabla\}$ consists of objects in $E\hbox{-}mod$. If moreover, $P':=\oplus P'_\alpha$ has no higher self-extension, we say the set $\{P'_\alpha\mid\alpha\in\nabla\}$ is the \textit{truncated mutation} of $\{P_\alpha\mid\alpha\in\nabla\}$ with respect to $P_{\theta}$, if the natural map $P_{\theta}\to \oplus_{\alpha\neq\theta }P_\alpha\otimes\Hom(P_{\theta},P_\alpha)_{t_\alpha}^*$ is injective.
\end{definition}

Whether truncated mutations exist or not, the object $P':=\oplus P'_\alpha$, considered as an object in $D^b(E\hbox{-}mod)$, always generates the triangulated category $D^b(E\hbox{-}mod)$, in the sense that $P'^\perp=0$ in $D^b(E\hbox{-}mod)$. This can be easily verified from the fact that $\oplus_\alpha P_\alpha$ generates $D^b(E\hbox{-}mod)$. Therefore, we have the following Lemma.

\begin{lemma}
If the truncated mutation exists, then we get an equivalence of derived categories $D^b(E\hbox{-}mod)\cong D^b(\End(P')\hbox{-}mod)$. Also in this case, the projective objects in the $t$-structure coming from $\End(P')\hbox{-}mod$ are objects in $E\hbox{-}mod$.
\end{lemma}

On the other hand, fixing a simple object $S_{\theta}$ in the abelian category $\catA$ such that $\Ext^1(S_\theta,S_\theta)=0$, we also have the tilting of $\catA$ with respect to $S_{\theta}$. Recall that the set of simple objects in $R_{S_{\theta}}\mathfrak{A}$ are given by Lemma~\ref{lem: simple}, and they are denoted by $\{S'_\alpha\mid\alpha\in\nabla\}$.
The $t$-structures obtained from truncated mutations and tiltings are related by the following Lemma.

\begin{lemma}\label{lem: t-struct-coincide}
Assume the truncated mutation of $\{P_\alpha\mid\alpha\in\nabla\}$ with respect to $P_{\theta}$ exists, and $\Ext^1(S_\theta,S_\theta)=0$. Then the $t$-structure obtained from $\End(P')\hbox{-}mod$ coincide with $R_{S_{\theta}}\catA$.
\end{lemma}
\begin{proof}
Two nested $t$-structures have to coincide. Therefore, it s enough to show that $\Ext^i(P'_\lambda,S'_\alpha)=0$ for all $\lambda$, $\alpha$ and all $i>0$.
Clearly, for all $\alpha$ and all $i>0$, we have $\Ext^i(P_\lambda,S'_\alpha)=0$ for all $\lambda\neq\theta$, and $Ext^i(P'_{\theta},S_{\theta}[1])=0$.
The only less clear point is the vanishing of $\Ext^1(P'_{\theta},S'_\alpha)$ for $\alpha\neq\theta$. For this we take the short exact sequence
$$0\to P_{\theta}\to \oplus_{\alpha\neq\theta}P_\alpha\otimes\Hom(P_{\theta},P_\alpha)_{t_\alpha}^*\to P'_{\theta}\to 0,$$ and look at the long exact sequence associated to it. Note that $\Hom(P_\lambda,S_\alpha)=\delta_{\lambda,\alpha}k$ and $\Ext^1(P_\lambda,S'_\alpha)=0$ for all $\lambda$, we get $$\cdots\to\Hom(P_\alpha,S'_\alpha)\otimes\Hom(P_{\theta},P_\lambda)_{t_\lambda}\to\Hom(P_{\theta},S_\alpha')\to\Ext^1(P'_{\theta},S'_\alpha)\to 0\to\cdots .$$
By the assumption that $\Hom(P_\theta,P_\alpha)$ maps surjectively to $\Ext^1(S_\alpha,S_\theta)^*$,
the map $$\Hom(P_\alpha,S'_\alpha)\otimes\Hom(P_{\theta},P_\lambda)_{t_\lambda}\to\Hom(P_{\theta},S_\alpha')$$ is also surjective, which conclude the vanishing of $\Ext^1(P'_{\theta},S'_\alpha)$.
\end{proof}

\begin{remark}
In fact, if we have more than one simple objects $S_1,\cdots,S_k$, we can define tilting with respect to all of them in a similar way. If $\Ext^1(S_i,S_j)=0$ for $i,j=1,\cdots,k$, then the tilted subcategory of the derived category is also a finite length abelian category, by the same argument. And if the truncated mutations exist, they also give the projective objects in the tilted subcategory.
\end{remark}

\subsection{Iterated tilting and iterated truncated mutation}

Fix a $\theta\in\nabla$ such that $\Ext^1(S_\theta,S_\theta)=0$. Recall that Proposition~\ref{prop: inter_tilt}says $R_{S_{\theta}[i-1]}R_{S_{\theta}[i-2]}\cdots R_{S_{\theta}}(\mathfrak{A})$ is still a finite length category. There is a construction of its simple objects, and they are denoted by $\{S_\alpha^i\mid\alpha\}$. Similarly, let $P^0_\alpha=P_\alpha$. Recursively we define $P_\theta^i$ to be the mapping one of the natural map $P_\theta^{i-1}\to \oplus_{\alpha\neq\theta }P^{i-1}_\alpha\otimes\Hom(P^{i-1}_{\theta},P^{i-1}_\alpha)_{t_\alpha}^*$. For $\alpha\neq \theta$, we define $P^i_\alpha$ to be $P_\alpha$.

\begin{lemma}\label{lem: ext-delta}
Notations as above, we have $\dim\Ext^k(P^i_\alpha,S^i_\beta)=\delta^\alpha_\beta $ for $k=0$ and vanishes for $k\neq0$.
\end{lemma}
\begin{proof}
We prove this by induction on $i$. For $i=0$, this is clear.

Assume the statement for $i-1$, now we show the corresponding statement for $k$. By chasing the $\Ext$ long exact sequence, we easily get $\Ext^k(P^{i-1}_\lambda,S^i_\alpha)=0$ for all $k\neq0$ and $\lambda\neq\theta$, and $\Ext^k(P^i_\theta,S^{i-1}_\theta[1])=0$ for $k\neq0$.

We only need to show $\Ext^k(P^i_\theta,S^i_\alpha)=0$ for $\alpha\neq\theta$. Looking at the $\Ext$ long exact sequence, this is equivalent to the surjectivity of $\Hom(P^{i-1}_\alpha,S^i_\alpha)\otimes\Hom(P^{i-1}_{\theta},P^{i-1}_\lambda)_{t_\lambda}\to\Hom(P^{i-1}_{\theta},S^i_\alpha)$.
Note also that $\Hom(P^{i-1}_\alpha,S^i_\alpha)\cong\Hom(P^{i-1}_\alpha,S^{i-1}_\alpha)$, and $\Hom(P^{i-1}_{\theta},S^i_\alpha)\cong\Hom(P^{i-1}_{\theta},S^{i-1}_\theta)\otimes\Ext^1(S_\alpha^{i-1},S_\theta^{i-1})^*$ for $\alpha\neq\theta$. This boils down to the surjectivity of $\Hom(P_\theta^{i-1},P_\alpha^{i-1})\to\Ext^1(S^{i-1}_\alpha,S^{i-1}_\beta)^*$.
\end{proof}

\begin{corollary}
Under the assumption of Lemma~\ref{lem: ext-delta}, if $E$ is a finite dimensional algebra, then truncated mutation with respect to $P_\theta$ exist as long as the natural map $P_{\theta}\to \oplus_{\alpha\neq\theta }P_\alpha\otimes\Hom(P_{\theta},P_\alpha)_{t_\alpha}^*$ is injective.
\end{corollary}

When we take the iterated mapping cone $P_\theta^i$, we assume that each time the truncated mutation exists. We know that $\End(\oplus P^i_\alpha)\hbox{-}mod$ is derived equivalent to $E\hbox{-}mod$, and $\{P^i_\alpha\mid\alpha\}$ is a set of projective generators in $\End(\oplus P^i_\alpha)\hbox{-}mod$. In particular, all projective object in it has a representative in $E\hbox{-}mod$, and the indecomposable projective objects are projective covers of the simple objects in $R_{S_{\theta}[i-1]}R_{S_{\theta}[i-2]}\cdots R_{S_{\theta}}(\catA)$.
Conversely, we have the following Lemma.
\begin{lemma}
Suppose $P_{\theta}$ is the projective cover of $S_{\theta}$ in $E\hbox{-}mod$ and the truncated mutation exists up to $i-1$ iterations. Assume $R_{S_{\theta}[i-1]}R_{S_{\theta}[i-2]}\cdots R_{S_{\theta}}(\catA)$ is of finite length with simple objects $\{S^i_{\alpha}\mid\alpha\in\nabla\}$, and the projective covers of them have representatives in $E\hbox{-}mod$. Then the truncated mutation $\{P_\alpha^i\mid\alpha\in\nabla\}$ exists.
\end{lemma}
\begin{proof}
We take the projective cover of $S^i_{\theta}$, denoted by $Q^i_{\theta}$, which can be chosen to be in $E\hbox{-}mod$. We know that $\Ext^j(Q^i_{\theta}, S_{\theta})$ vanish for $j\neq i$ and is one dimensional when $j=i$.

We take the minimal projective resolution of $Q^i_{\theta}$ in $\End(\oplus P^{i-1}_\alpha)\hbox{-}mod$. It has length 2 as the projective dimension of $Q^i_{\theta}$ is 1.
The degree 1 term of the resolution has $P_{\theta}^{i-1}$ as a summand and the degree 0 term does not have summand $P_{\theta}^{i-1}$. This already implies the injectivity of $P_{\theta}^{i-1}\to\oplus_{\alpha\neq\theta}\Hom(P_{\theta}^{i-1},P_\alpha)^*_t\otimes P_\alpha$.
\end{proof}

\begin{example}
Let $\catA$ be the category of perverse sheaves on $\PP^n$ with the standard stratification. Let $S_n$ be the simple object $\CC_{\PP^n}[n]$ which is an $\PP^n$ object in this category. The semi-reflection of $D^b(\catA)$ with respect to $S_n$ can be obtained by taking the image of $\Perv(\PP^{n*})$ under the Radon transform. In particular, the semi-reflection is derived equivalent to $\catA$ and equivalence comes from a tilting generator in $\catA$. In fact, according to Proposition~\ref{prop: Chuang-Rouquier}, if one do tilting with respect to $S_n$ for $n$ times, one will get the same $t$-structure as the semi-reflection.

We will illustrate Proposition~\ref{prop: exist_trunc_mutat} by explicitly calculation of the tilting generator for the intermediate $t$-structures, i.e., those obtained from tilting with respect to $S_n$ for $i$ times, for any $i<n$.

For simplicity, we take $n=2$. The general case is similar. The category $\catA$ is Morita equivalent to the module category of the quiver
$$\xymatrix{
\bullet_{pt}\ar@/^/@<2ex>[r]|-{\alpha}&\bullet_{\Aff^1}\ar@/^/@<2ex>[r]|-{\delta}\ar@/^/@<2ex>[l]|-{\beta}&\bullet_{\Aff^2}\ar@/^/@<2ex>[l]|-{\gamma}
}$$
with relations $\alpha\beta=0$, $\delta\gamma=0$, $\delta\alpha=0$, and $\beta\gamma=0$.
The projective objects in $\catA$ are
\[P_{pt}=\xymatrix{\CC^2_{pt}\ar@/^/@<2ex>[r]|-{\alpha}&\CC_{\Aff^1}\ar@/^/@<2ex>[l]|-{\beta}};\]
\[P_{\Aff^1}=\xymatrix{\CC_{pt}&\CC^2_{\Aff^1}\ar@/^/@<2ex>[l]|-{\beta}\ar@/^/@<2ex>[r]|-{\delta}&\CC_{\Aff^2}\ar@/^/@<2ex>[l]|-{\gamma}};\]
\[P_{\Aff^2}=\xymatrix{\CC_{\Aff^1}&\CC_{\Aff^2}\ar@/^/@<2ex>[l]|-{\gamma}}.\]

 We consider the tilting with respect to $S_2$:
The tilting generators are: $P_{pt}$, $P_{\Aff^1}$, and  \[\coker(P_{\Aff^2}\to P_{\Aff^1})\cong P'_{\Aff^2}=\xymatrix{\CC_{pt}&\CC_{\Aff^1}\ar@/^/[l]|-{\beta}}.\]

Then we consider the tilting with respect to $S_2[1]$:
The tilting generators are: $P_{pt}$, $P_{\Aff^1}$, and \[\coker(P'_{\Aff^2}\to P_{pt})\cong P_{\Aff^2}''=\CC_{pt}.\]

The hearts of all these $t$-structures are derived equivalent to $\catA$.

If we do tilting with respect to $S_2$, the tilting generators of the new heart will be $P_{\Aff^1}$, $P_{\Aff^2}$, and the cokernel of the map $P_{\Aff^2}\to P_{\Aff^1}$ which is $P'_{\Aff^2}=\xymatrix{\CC_{\Aff^1}&\CC_{\Aff^2}\ar@/^/[l]|-{\beta}}$. If we do tilting another time with respect to $S_2[1]$, the tilting generators of the new heart will be $P_{\Aff^1}$, $P_{\Aff^2}$, and the cokernel of the map $P'_{\Aff^2}\to P_{pt}$ which is $P_{\Aff^2}''=\CC_{pt}$. The hearts of all these $t$-structures are derived equivalent to $\catA$.
\end{example}

\subsection{Truncated mutations from geometric origin}\label{subsec: truncate_geom}
Now let $\bfX$ be a smooth variety which is projective over $\Spec A$. Also we assume the map $\pi: \bfX\to \Spec A$ is $\Gm$-equivariant, such that $\bfX$ is deformation retracts to $X=\pi^{-1}(\Spec A/m)$, the fiber over $A/m$ under this $\Gm$-action.
Let $\{P_\alpha\mid\nabla\}$ be a collection of $\Gm$-equivariant vector bundles on $\bfX$, which classically generates $\Qcoh(\bfX)$ and $\Ext^i(\oplus P_\alpha,\oplus P_\alpha)=0$ for all $i>0$. Let $E=\End(\oplus_{\alpha\in \nabla}P_\alpha)$. Then \cite{BV02} gives a equivalence of derived categories $D(\Qcoh(\bfX))\cong D(E\hbox{-}Mod)$, and it restricts to equivalences $D^b(E\hbox{-}mod)\cong D^b(\Coh(\bfX))$, and $D^b_{A/m}(E\hbox{-}mod)\cong D^b_X(\Coh(\bfX))$. Now we take $\catA$ to be the category of $E$-modules which are set-theoretically supported at $A/m$.

Fix a $\theta\in \nabla$. Assume $S_\theta$ is a simple object in $\catA$ with $\Ext^1(S_\theta,S_\theta)=0$. We take $P^0_\alpha$ to be $P_\alpha$ for all $\alpha\in\nabla$. Recursively, we define $P^i_\theta$ to be the mapping cone of the natural map $P_{\theta}^{i-1}\to\oplus_{\alpha\neq\theta}\Hom(P_{\theta}^{i-1},P_\alpha)^*_t\otimes P_\alpha$, and  $P^i_\alpha=P^{i-1}_\alpha$ for $\alpha\neq\theta$.
Define $P^i=\oplus_{\alpha\in\nabla} P^i_\alpha$.

\begin{lemma}
Assume $S_\theta$ is a simple object in $\catA$ with $\Ext^1(S_\theta,S_\theta)=0$.
The map $\Hom(P_\theta^{i-1},P_\alpha^{i-1})\to\Ext^1(S_\alpha,S_\beta)^*$ induced by the composition morphism $\Hom(P_\alpha,S_{\alpha,\theta})\otimes\Hom(P_\theta,P_\alpha)\to\Hom(P_\theta,S_{\alpha,\theta})$ is surjective.
\end{lemma}
\begin{proof}
For a complex $N$, let $\red N$ be the complex fit in the exact triangle $\red N\to N\to \oplus_{\alpha}\Hom(N,S^{i-1}_\alpha)\otimes S^{i-1}_\alpha$. Then we have $\Ext^i(\red N,S^{i-1}_\alpha)=0$ for all $i<0$ if this property holds for $N$.

We have, from the exact triangle $\red P^{i-1}_\alpha\to P_\alpha^{i-1}\to S_\alpha^{i-1}$, that $\Hom(P_\theta^{i-1},P_\alpha^{i-1})\cong \Hom(P_\theta^{i-1},\red P_\alpha^{i-1})$. Also from the exact triangle $\red\red P_\alpha^{i-1}\to\red P_\alpha^{i-1}\to \oplus_{\alpha}\Hom(\red P_\alpha^{i-1},S^{i-1}_\alpha)\otimes S^{i-1}_\alpha$, we have $\Hom(\red P_\alpha^{i-1},S_\theta^{i-1})\cong\Hom(P_\theta^{i-1},\oplus_{\alpha}\Hom(\red P_\alpha^{i-1},S^{i-1}_\alpha)\otimes S^{i-1}_\alpha)^*$. We only need to show $$\Ext^1(P_\theta^{i-1},\red\red P_\alpha^{i-1})=0.$$

For this purpose, note that the complex $Q\cong\red\red P_\alpha^{i-1}$ can be chosen $\Gm$-equivariantly. Let $Q_k$ be $Q/m^kQ$. Then $Q$ can be obtained by taking the $\Gm$ finite part of $\underleftarrow{\lim}Q_k$. Since $P_\theta^{i-1}$ is equivariant under $\Gm$, we have the canonical isomorphism of complexes $R\Hom(P_\theta^{i-1},Q)\cong R\Hom(P_\theta^{i-1},\underleftarrow{\lim}Q_k)\cong \underleftarrow{\lim}R\Hom(P_\theta^{i-1},Q_k)$.
Hence, $H^1(R\Hom(P_\theta^{i-1},Q_k))=0$ implies $H^1(R\Hom(P_\theta^{i-1},Q))\cong\Ext^1(P_\theta^{i-1},Q)=0$. Note that $Q_k$ lies in $D^b(\catA)$, and has the property that $\Ext^i(Q_k,S^{i-1}_\alpha)=0$ for all $i<0$, all $\alpha$, and large enough $k$.
This means $Q_k$ can be chosen as a complex concentrated in non-positive degrees with respect to the $t$-structure $R_{S_{\theta}[i-2]}R_{S_{\theta}[i-3]}\cdots R_{S_{\theta}}(\catA)$. Therefore, we have $\Ext^1(P_\theta^{i-1},Q_k)=0$.
\end{proof}

Take $P'_\theta$ to be the mapping cone of the natural map $P_{\theta}\to \oplus_{\alpha\neq\theta }P_\alpha\otimes\Hom(P_{\theta},P_\alpha)_{t_\alpha}^*$, and $P'_\alpha=P_\alpha$ for $\alpha\neq\theta$.
Then, there is a equivalence between $D(E\hbox{-}mod)$ and $D(R\Hom(P',P'))$, where $R\Hom(P',P')$ is understood as a DG-algebra.
The DG-algebra $R\Hom(P',P')$ has homology concentrated in non-negative degrees not exceeding 1, and is concentrated in degree zero if and only if $P_{\theta}\to \oplus_{\alpha\neq\theta }P_\alpha\otimes\Hom(P_{\theta},P_\alpha)_{t_\alpha}^*$ is injective, namely, the truncated mutation exits.
In general, we also have an equivalence $D(E\hbox{-}mod)\cong D(R\Hom(P^i,P^i))$.
Inductively, the DG-algebra $R\Hom(P^i,P^i)$ has homologies concentrated in non-negative degrees, and is concentrated in degree zero if and only if $P^{j-1}_{\theta}\to \oplus_{\alpha\neq\theta }P^j_\alpha\otimes\Hom(P^{j-1}_{\theta},P^{j-1}_\alpha)_{t_\alpha}^*$ is injective for all $j\leq i$.

\begin{prop}\label{prop: exist_trunc_mutat}
Assume $S_\theta$ is a simple object in $\catA$ with $\Ext^1(S_\theta,S_\theta)=0$. Assume the $n$-th iterated tilting with respect to $S_\theta$ has a set of indecomposable projectives $\{Q_\alpha\}$ consists of objects concentrated in degree zero. Then the iterated truncated mutations up to $n$ times exist.
\end{prop}
\begin{proof}
As we have $\Hom(P^i_\alpha, S^i_\beta)=\delta^\alpha_\beta k$ according to Lemma~\ref{lem: ext-delta}. This means $P^i_\alpha\cong Q_\alpha$ for all $\alpha$, and hence $P^i_\alpha$ is concentrated in degree zero.
\end{proof}

\begin{remark}
If $\catA$ is a finite length abelian category with enough projective objects, then the conclusion in  Proposition~\ref{prop: exist_trunc_mutat} still holds.
\end{remark}

\begin{corollary}\label{cor:perver-smaller_pervers}
Assume $S_\theta$ is a simple object in $\catA$ with $\Ext^1(S_\theta,S_\theta)=0$. We endow $\catA$ with the filtration that $0=\catA_0\subseteq\catA_1\subseteq\catA_2=\catA$ where $\catA_1=\langle S_\theta\rangle$. Assume for the perversity function $p$ with $p(1)=0$ and $p(2)=n$ we have a perverse equivalence $(t,t',p)$ such that the projective covers of the simple objects in the heart of $t'$ have representatives lying in $E\hbox{-}mod$. Then for any $p'$ with $p'(1)=0$ and $p'(1)\leq n$ the perverse equivalence $(t,t'',p')$ exists, and the projective covers of the simple objects in the heart of $t'$ have representatives lying in $E\hbox{-}mod$.
\end{corollary}

A typical example of truncated mutations from geometric origin is the following one. 

\begin{example}
Let $\pi:T^*\PP^n\to \PP^n$ and let $D=D^b(\Coh_0T^*\PP^n)$.
Let $\catA=$ heart of the $t$-structure in $D$ induced by the tilting bundle  $\pi^*(\oplus_{i=0}^n\calO(i))$ on $T^*\PP^n$. The 
 simple objects in $\catA$ are $\{\wedge^i\calQ^*|_{\PP^n}[i]\mid i=0,\dots,n\}$.
 In $D^b(\Coh_0(T^*\PP^{n\vee}))$ there is a $t$-structure $\catA'$ induced by the tilting object $\pi^*(\oplus_{i=0}^n\calO_{\PP^{n\vee}}(i))$ on $T^*\PP^{n\vee}$.
The transform of $\catA'$ under the Fourier-Mukai transform of Namikawa in \cite{Nam} is the semi-reflection of $\catA$ with respect to $S:=\wedge^n\calQ^*|_{\PP^n}[n]$.   Clearly $S$ is a $\PP^n$-object.

\[\xymatrix{
  T^*\PP^n\ar[dr]&&T^*(\PP^{n\vee})\ar[dl]\\
  &\calN&
  }\]
  
According to Proposition~\ref{prop: exist_trunc_mutat}, this $t$-structure can alternatively be described as iterative tilting with respect to $S$ n-times. We study the projective generators in the hearts of all these intermediate $t$-structures using truncated mutation.

For simplicity, we take $n=2$. The algebra $\End_{T^*\PP^2}(pi^*(\oplus_{i=0}^2\calO(i)))$ can be described by the following quiver (we are following the conventions in \cite[\S5]{WZ})
$$\xymatrix{
\bullet_{0}\ar@/^/@<2ex>[r]^{\alpha(\CC^3)}&\bullet_{1}\ar@/^/@<2ex>[r]^{\delta(\CC^3)}\ar@/^/@<2ex>[l]^{\beta(\CC^{3*})}&\bullet_{2}\ar@/^/@<2ex>[l]^{\gamma(\CC^{3*})}
}$$
with relations
\[\delta\alpha(\wedge^2\CC^3);\ \beta\gamma(\wedge^2\CC^{3*}); \]
\[\beta\alpha(\CC);\ \delta\gamma(\CC);\ \gamma\delta+\alpha\beta(\CC).\]
The projective objects, $P_0$, $P_1$, and $P_2$ are spanned by  paths starting at the vertices $0$, $1$, and $2$ respectively.

Consider the $t$-structure obtained by tilting of $\catA$ with respect to $S$.  The indecomposable  projective objects are $P_0$, $P_1$, and $\tilde{P_2}$, where $\tilde{P_2}$ is the mapping cone of the morphism $P_2\to P_1\otimes\CC^3$. It can be visualized as pre-composing paths from $1$ with the arrow $\delta$. As the relation indicates, the morphism $P_2\to P_1\otimes\CC^3$ is injective. In terms of the quiver picture, this fact is equivalent to that pre-composing with $\delta$ does not kill any path from $1$. Therefore, $\tilde{P_2}$ is the cokernel of $P_2\to P_1\otimes\CC^3$. In terms of quivers, $\tilde{P_2}$ is spanned by paths  from $1$ that does not have $\delta$ as its first arrow.

Consider the $t$-structure obtained by tilting of $R_S\catA$ with respect to $S[1]$. The indecomposable  projective objects are $P_0$, $P_1$, and $\tilde{\tilde{P_2}}$, where $\tilde{P_2}$ is the mapping cone of the morphism $\tilde{P_2}\to P_0\otimes\wedge^2\CC^3$. In terms of quivers, this map can be visualized as pre-composing paths from $1$ with the arrow $\alpha$. Again it is easy to see that this map is injective, hence $\tilde{\tilde{P_2}}$ is the cokernel of $\tilde{P_2}\to P_0\otimes\wedge^2\CC^3$. In terms of quivers, $\tilde{\tilde{P_2}}$ is spanned by paths from $0$ that do not have $\alpha$ as its first arrow. The only such path is the constant path at $0$. 
To summarize, this $t$-structure is the semi-reflection of $\calA$ with respect to the $\PP^2$-object $S$. The indecomposable projective objects in the semi-reflection are $P_0$, $P_1$, and a quotient of $P_0\otimes\wedge^2\CC^3$.

\end{example}

Another example of truncated mutations from geometric origin as in the set-up of this subsection will be given in Section~\ref{sec: t-struct_alco}.

\subsection{Koszulity of truncated mutations}
\begin{lemma}
In the set up of Section~\ref{subsec: truncate_geom}, assume there is a choice of $\{S_\alpha\mid\alpha\in\nabla\}$ such that each one is graded, and $\Ext^1(S_\alpha,S_\beta)$ has homogeneous degree one for any $\alpha$ and $\beta\in \nabla$. Then there is such a choice for $\{S_\alpha'\mid\alpha\in\nabla\}$ with the same properties.
\end{lemma}
\begin{proof}
We define the grading on $\{S_\alpha'\mid\alpha\in\nabla\}$ as follows. For $S'_\theta\cong S_\theta[1]$, we define the its degree to be the degree of $S_\theta$ -1. For $\alpha\neq\theta$, we define the degree of $S_\alpha'$, which is the universal extension of $S_\alpha$ by $S_\theta$, by keeping the degree of $S_\alpha$ and $S_\theta$ as they are, and (twist the original grading) declare $\Ext^1(S_\alpha,S_\theta)^*$ to be in degree zero.

Then we immediately get that $\Ext^1(S_\theta[1],S'_\alpha)\cong\Hom(S_\theta,S_\theta)\otimes\Ext^1(S_\alpha,S_\theta)^*$ has degree 1, since $\Hom(S_\theta,S_\theta)$ has degree 1 and $\Ext^1(S_\alpha,S_\theta)^*$ has degree zero.

As for $\Ext^1(S'_\alpha,S_\theta[1])\cong\Ext^2(S'_\alpha,S_\theta)$, where $\alpha\neq\theta$, look at the following part of a long exact sequence
\[\cdots\to\Ext^2(S_\alpha,S_\theta)\to\Ext^2(S'_\alpha,S_\theta)\to\Ext^2(S_\theta,S_\theta)\otimes\Ext^1(S_\alpha,S_\theta)\to\cdots,\]we need to show the terms at two sides both have degree 1. For $\Ext^2(S_\alpha,S_\theta)$ this is clear by assumption and the fact that the degree of $S_\theta$ has been reduced by 1. For the same reason, $\Ext^2(S_\theta,S_\theta)$ also has degree 1. The degree of $\Ext^1(S_\alpha,S_\theta)$ has been declared to be zero. So, the term $\Ext^2(S_\theta,S_\theta)\otimes\Ext^1(S_\alpha,S_\theta)$ also has degree one.

In the case $\alpha$, $\beta\neq\theta$, we first show that $\Ext^1(S'_\beta,S_\alpha)$ has degree 1. This can be done by observing the two sides of the following part of a long exact sequence $$\cdots\to\Ext^1(S_\beta,S_\alpha)\to\Ext^1(S'_\beta,S_\alpha)\to\Ext^1(S_\theta,S_\alpha)\otimes\Ext^1(S_\beta,S_\theta)\to\cdots.$$ Then we look at the following part of a different long exact sequence $$\cdots\to\Ext^1(S'_\beta,S_\theta)\otimes\Ext^1(S_\alpha,S_\theta)^*\to\Ext^1(S_\beta',S_\alpha')\to\Ext^1(S'_\beta,S_\alpha)\to\cdots.$$ Note that $\Ext^1(S_\alpha',S_\theta)=0$. We conclude that $\Ext^1(S_\beta',S_\alpha')$ also has degree 1.
\end{proof}
\begin{lemma}
Assume there is a choice for $\{P_\alpha\mid\alpha\in\nabla\}$ such that each $P_\alpha$ is graded and $\End(P)$ has only non-negative degree pieces, with $\Hom(P_\theta,P_\alpha)_{t_\alpha}$ lies in homogeneous degree one. Assume further that the truncated mutation exists, then there is also such a choice for $\{P'_\alpha\mid\alpha\in\nabla\}$ such that $\End(P')$ has only non-negative degree pieces, and $\Hom(P'_\theta,P'_\alpha)_{t_\alpha}$ can be chosen to be in homogeneous degree one.
\end{lemma}
\begin{proof}
We define the grading on $\{P_\alpha'\mid\alpha\in\nabla\}$ as follows. For $\alpha\neq\theta$, we have $P'_\alpha\cong P_\alpha$, and we keep the grading of it as it is. For $P'_\theta$ which is the cokernel of the map $P_{\theta}\to \oplus_{\alpha\neq\theta }P_\alpha\otimes\Hom(P_{\theta},P_\alpha)_{t_\alpha}^*$, we use the grading of $P_\theta$ and $P_\alpha$ and declare $\Hom(P_{\theta},P_\alpha)_{t_\alpha}^*$ to be in degree -1.

Clearly, $\Hom(P'_\alpha,P'_\beta)$ for $\alpha$, $\beta\neq\theta$ has not been influenced. Also it is clear that $\Hom(P'_\theta,P_\alpha')$ has non-negative grading. Note also that $\Hom(P'_\theta,P'_\theta)$ embeds into $\oplus_{\alpha\neq\theta}\Hom(P_\alpha,P_\theta')\otimes\Hom(P_\theta,P_\alpha)_{t_\alpha}$. We only need to show the non-negativity of the grading of $\Hom(P_\alpha,P'_\theta)$ for $\alpha\neq\theta$. For this, we need to show the degree -1 part of $\Hom(P_\alpha,P_\alpha)\otimes\Hom(P_\theta,P_\alpha)_{t_\alpha}^*$ maps injectively into $\Hom(P_\alpha,P_\theta)$. This is clear from the construction of $\Hom(P_\theta,P_\alpha)_{t_\alpha}^*$.
\end{proof}

\section{The $t$-structures from quantization in positive characteristic}\label{sec: quant_char_p}
\subsection{Localization of rational Cherednik algebras}
We work over a separably closed field $k$  of characteristic $p$ which is large enough.
Let $\Gamma_1\subseteq \SL(2)$ be a finite subgroup.
Let $\Gamma_n:=(\Gamma_1)^n\rtimes\mathfrak{S}_n$ acting on $\Aff^{2n}\cong(\Aff^2)^n$ in the natural way, i.e., the $i$-th copy of $\Gamma_1$ acts on the $i$-th $\Aff^2$ summand, and $\mathfrak{S}_n$ permutes the coordinates. There is a natural symplectic form on $\Aff^{2n}$. It is preserved by the diagonal action of $\Gamma_n$.
A symplectic resolution of $\Aff^{2n}/\Gamma_n$ can be given as $\Hilb^n(\widetilde{\Aff^2/\Gamma_1})$, where $\widetilde{\Aff^2/\Gamma_1}$ is the minimal resolution of $\Aff^2/\Gamma_1$. Leter on we will use the short hand notation $\Hilb^n_{\Gamma_1}$ for $\Hilb^n(\widetilde{\Aff^2/\Gamma_1})$ or simply $\Hilb^n$ when $\Gamma_1$ is clear from the context.

It is well-known (see \cite{Kuz01}) that a symplectic resolution of $\Aff^{2n}/\Gamma_n$ can be constructed as a Nakajima quiver variety of extended Dynkin quiver with suitable dimension vectors and stability conditions. Recall that the Nakajima variety of a quiver $Q$ with dimension vectors $v$ and $w$ and stability condition $\theta$ is the Hamiltonian reduction $T^*(\hbox{Rep}(Q,v)\oplus\Hom(k^v,k^w))/\!/_\theta GL(v)$.
For suitable choice of stability condition, the Nakajima variety is isomorphic to $\Hilb^n(\widetilde{\Aff^2/\Gamma_1})$. In particular, we know $H^2(\Hilb^n(\widetilde{\Aff^2/\Gamma_1}))$ is isomorphic to the character group of $GL(v)$, which is a free abelian group with a basis indexed by the vertices of this quiver. The Weil divisors on $\Hilb^n$ corresponding to these basis elements are in turn in natural one to one correspondence with the congugacy classes of symplectic reflections in the group $\Gamma_n$ (see, e.g., \cite[\S 4]{BK04} for a description of this correspondence).

Write $\Hilb^{n(1)}$ for the Frobenius twist $\Hilb^n(\widetilde{\Aff^2/\ZZ_l})^{(1)}$.
Quantizations of $\Hilb^{n(1)}$ are related to rational Cherednik algebras. The precise relationship is given by \cite{BFG06} which we briefly summarize below for the convenience of the readers.

Let $\hbox{Ref}$ be the set of reflections in $\Gamma_n$.
Decompose $\hbox{Ref}=\coprod_{i=0}^r \hbox{Ref}_i$ into conjugacy classes.
Pick integers $c=(c_0, c_1,\cdots, c_r)$, the \textit{rational Cherednik algebra} is defined to be \[H_c=H_c(\hh,\Gamma_n):=k[\hh]\langle\hh^*\rangle\#\Gamma_n/I\] where $I$ is the two-sided ideal generated by $[u, v] = \langle u, v\rangle - 2\sum_{i=1}^r c_i\sum_{\gamma\in \hbox{Ref}_i}\langle u, v\rangle_\gamma\cdot\gamma$ for $u\in\hh$ and $v\in \hh^*$, where $\langle-,-\rangle_\gamma$ is the paring between $\im(\gamma-1)$ and its dual.

The algebra $H_c$ has a natural filtration, and the associated graded algebra is $k[\Aff^{2n}]\#\Gamma_n$.
Let $\Aff^{2n(1)}$ be the Frobenius twist of $\Aff^{2n}$, then the algebra $H_c$ has a big \textit{Frobenius center} $k[\Aff^{2n(1)}]^\Gamma_n$.
For any central character $\chi$, (i.e, an element in the maximal spectrum of $k[\Aff^{2n(1)}]^\Gamma_n$,) we can consider the category of finitely generated modules over $H_c$, on which the Frobenius center acts by the central character $\chi$. This category will be denoted by $\Mod_\chi H_c$.
The irreducible objects in the category $\Mod_\chi H_c$ are naturally labeled by elements in $\hbox{Irrep}(\Gamma_n)$.

Let  $e:=\sum_{\gamma\in \Gamma_n} \gamma$. For generic values of $c\in\hbox{Span}_\QQ\hbox{Ref}$, the algebra $H_c$ is Morita equivalent to $^sH_c:=eH_ce$. If there is a Morita equivalence, the value $c$ is said to be a \textit{spherical value}. The special values are called \textit{aspherical values}.
Let $\Mod_0\null^sH_c$ be the category of $^sH_c$-modules with central character $0$. The irreducible object in $\Mod_0\null^sH_c$ labeled by $\tau \in \hbox{Irrep}(\Gamma_n)$ will be denoted by $L_c(\tau)$.

Taking any $\chi\in H^2(\Hilb^n,\QQ)$, there is a quantization $\sA_\chi$ of $\Hilb^{n(1)}$ coming from the quantum Hamiltonian reduction of the sheaf of $\chi$-twisted differential operators on $\hbox{Rep}(Q,v)\oplus\Hom(k^v,k^w)$.

\begin{theorem}[\cite{BFG06}]
For each $c$, there is a sheaf of algebras $\sA_c$ on $\Hilb:=\Hilb^n(\widetilde{\Aff^2/\Gamma_1})$, which is an Azumaya algebra on $\Hilb^{(1)}$. It has the following properties.
\begin{enumerate}
  \item The Azumaya algebra $\sA_c$ splits on the formal neighborhood of the fibers of the Hilbert-Chow morphism;
  \item $H^i(\Hilb^{(1)},\sA_c)=0$ for $i>0$;
  \item for large enough $p$, one has an isomorphism
  $$\phi_c: \Gamma(\Hilb^{(1)},\sA_c)\cong \null^sH_c;$$
  \item for spherical values $c$, $^sH_c$ has finite global dimension, in which case there is a derived equivalence $D^b(\Coh_0\Hilb^{(1)})\cong D^b(\Mod_0\null^sH_c)$.
\end{enumerate}
\end{theorem}
In the terminology of \cite{BO}, the algebras $^sH_c$'s  are \textit{noncommutative resolutions of singularities} when $c$ is spherical. They are derived equivalent to $\Coh(\Hilb^{(1)})$.
As the splitting vector bundle on the formal neighborhood can be chosen to be $\Gm$-equivarient, therefore, a standard argument   shows that it extends to a vector bundle $\sE_c$ on the entire $\Hilb\frob$, and induces a global derived equivalence between $\Coh_0(\Hilb^{(1)})$ and $\Mod_c\End(\sE_c)$.

In particular, take $c=0$ (which is always spherical) we get a derived equivalence $$D^b(\Coh\Hilb^n_{\Gamma_1})\cong D^b(\Coh_{\Gamma_n}(\Aff^{2n})).$$ This equivalence is called the symplectic McKay correspondence. The splitting bundle $\sE_0$ has the same indecomposable summands as the Procesi bundle studied in \cite{H02} and \cite{Los}.

For each spherical value $c$, the derived equivalence $$D^b(\Coh_0\Hilb^{(1)})\cong D^b(\Mod_0\null ^sH_c)$$ endows $D^b(\Coh_0\Hilb^{(1)})$ with a $t$-structure, whose heart is the image of $\Mod_0\null^sH_c$ under this equivalence.

\begin{question}
For two different spherical values $c$ and $c'$, what is the relation between the $t$-structures on  $D^b(\Coh_0\Hilb^{(1)})$?
\end{question}

The aspherical values form a union of affine hyperplanes. The open facets will be  called \textit{alcoves}, and codimension-1 facets will be called \textit{walls}.
If $c$ and $c'$ are in the same alcove, then the translation functor induces a Morita equivalence; the $t$-structures are the same. In particular, the aspherical values are exactly the locus where the central charge applied to some simple object vanishes.

Let $L_c(\tau)$ be the irreducible object in $\Mod_0\null^sH_c$ labeled by $\tau \in \hbox{Irrep}(\Gamma_n)$.
Under the derived equivalence of \cite{BFG06},
for any irreducible object $L_c(\tau)\in\Mod_0\null ^sH_c$, let the   corresponding complex in $D^b(\Coh_0\Hilb\frob)$ be denoted by $\sL_c(\tau)$; for the projective cover of $L_c(\tau)$ in $\Mod\End(\sE_c)$, let the corresponding vector bundle on $\Hilb$ be denoted by $\sV_\tau$.  We have \[\Ext^i(\sV_\alpha,\sL_\beta)=\left\{
                            \begin{array}{ll}
                              \delta_{\alpha,\beta}, & \hbox{$i=0$;} \\
                              0, & \hbox{$i>0$.}
                            \end{array}
                          \right.\]

As a corollary of the derived localization theorem, the translation functors, and the Hirzebruch-Riemann-Roch, we get the following Lemma.
\begin{lemma}\label{lem: dim_poly}
When $c+\nu$ is in the same alcove as $c$, $$\dim L_{c+\nu}(\tau):=\chi(\sL_c(\tau)\otimes\sE_0\otimes\sO(\nu))$$ is a polynomial in $\nu$. Here for any $\nu\in H^2(\Hilb^{n(1)})$, the corresponding line bundle is denoted by $\sO(\nu)$.
\end{lemma}
These polynomials will be referred to as the dimension polynomials.

We define \begin{equation}\label{eq:Z}Z_\tau(c)=\lim_{p\to\infty}p^{-n}\dim_kL_{cp}(\tau;p).\end{equation}

We consider the collection of polynomials $\{Z_\tau(c)\mid\tau\in\hbox{ Irrep}(\Gamma_n)\}$ as a polynomial map \[H^2(\Hilb;\bbQ)\to \Hom_{\ZZ}(K_0(\Hilb),\bbQ).\] This polynomial map is called {\it the central charge}.

We will make precise of the slogan that the central charge controls the difference of the $t$-structures associated to neighboring alcoves.

\subsection{Real variation of stability conditions}

Bridgeland introduced a notion of stability conditions (see \cite{Bri}) which parameterizes all bounded $t$-structures of the same triangulated category and goes along well with deformations.
Recall that for an abelian category $\mathfrak{A}$, a stability function on it is a group homomorphism $Z: K(\mathfrak{A})\to \mathbb{C}$ such that $$0\neq E\in \mathfrak{A}\Rightarrow Z(E)\in\mathbb{R}_{>0}\exp(i\pi\phi(E))$$ where the real number $\phi(E)\in(0,1]$ is called the phase of $E$. A nonzero subobject is said to be semi-stable with respect to $Z$ if every subobject has smaller or equal phase. The stability function $Z$ is said to have the Harder-Narasimhan property if every nonzero object $E\in \catA$ has a finite filtration $0 = E_0 \subset E_1 \subset \cdots \subset E_n = E$
whose factors $F_j = E_j/E_{j-1}$ are semistable objects of $\catA$ with $\phi(F_1) > \phi(F_2) > \cdots > \phi(F_n)$.
\par
Defining a Bridgeland stability condition on a triangulated category $D$ is equivalent to giving a bounded $t$-structure on $D$ together with a stability function on its heart with the Harder-Narasimhan property.
Bridgeland showed that the set $\Stab(D)$ of all stability conditions on a triangulated category $D$ has a complex manifold structure, such that the function $\Stab(D)\to K(D)_{\CC}^*$ sending any stability condition to its stability function is an local isomorphism to a subspace $V\subseteq K(D)_{\CC}^*$ on each connected component of $\Stab(D)$.
\par
There is a notion of real variation of stabilities defined in \cite{ABM}, based on similar idea as in the definition of Bridgeland. We briefly recall the definition here.

Let $D$ be a $k$-linear triangulated category with finite rank $K$-group and finite dimensional $\Hom$'s, and $V$ a real vector space.
Fix a discrete collection $\Sigma$ of affine hyperplanes in $V$. Let $V^0$ denote their
complement.
Let $\Sigma_{lin}$ be the set of their translations through zero, a collection of linear hyperplanes. Fix a component
$V^+$ of $V\backslash\cup\Sigma_{lin}$. The choice of $V^+$
determines for each $H\in \Sigma$ the choice of the positive half-space $(V\backslash H)^+\subset  V\backslash H$. Let $Alc$ denote the set of all alcoves, i.e., connected components of $V^0$.
\begin{definition}
A \textit{real variation of stability conditions} on $D$ parametrized by $V^0$ and
directed to $V^+$ is the data $(Z, \tau)$, where $Z$ (the central charge) is a polynomial map
$Z : V \to(K^0(D)\otimes\RR)^*$, and $\tau$ is a map from $Alc$ to the set of bounded $t$-structures
on $D$ with finite length hearts, subject to the following conditions.
\begin{enumerate}
  \item For $0\neq M\in \tau(A)$ and $x\in A$, $\langle Z(x), [M]\rangle> 0$.
  \item Suppose $A, A'\in Alc$ share a codimension one face and $A'$ is above $A$.
Let $A_n \subseteq\tau(A)$ be the full subcategory $\{M \in A_n\mid \langle Z(x), [M]\rangle\hbox{ has zero of
order at least }n\}$. Then we require:
\begin{itemize}
  \item The $t$-structure $\tau(A')$ is compatible with the filtration.
  \item The $t$-structure on $\gr_n(D) = D_n/D_{n+1}$
induced by $\tau(A)$ differers from  that of $\tau(A')$ by $[n]$.
\end{itemize}
\end{enumerate}
\end{definition}

Now we can state our main theorem of this paper.
\begin{theorem}\label{thm: main}
Let $Z:H^2(\Hilb^2_{\ZZ_l};\QQ)\to K_\QQ(\Hilb^2_{\ZZ_l})^\vee$ be defined as in \eqref{eq:Z}. Let $\tau$ be the assignment associating each alcove in $H^2(\Hilb^2_{\ZZ_l};\QQ)$ the $t$-structure on $D^b(\Coh_0(\Hilb^2_{\ZZ_l}))$ whose heart is given by $\Mod_0 \null^sH_c$ for some $c$ in this alcove.
Then, the pair $(Z,\tau)$ is a real variation of stability conditions on $D^b(\Coh_0(\Hilb^2_{\ZZ_l}))$.
\end{theorem}

\subsection{Comparison with category $\sO_c$ in characteristic zero}\label{sec: compare_char0}

Let $\Gamma$ be an arbitrary reflection group acting on $\hh$. Let $V=\hh\oplus\hh^*$ with the natural symplectic form and the diagonal $\Gamma$-action.

Let $R$ be a $\ZZ$-subalgebra of $\C$, finitely generated over $\bbZ$, such that $H_c(\Gamma_2)_R$ exists.
Let $\calO_c$ be the category $\calO$ of $H_c(\Gamma)_\C$. For any $\tau\in\hbox{Irrep}(\Gamma)$ let the corresponding irreducible object over $^sH_c(\Gamma)_\C$ be denoted by $L_c(\tau;\CC)$, and its $R$-form be $L_c(\tau;R)$. Let $\overline{L_c(\tau;R)_k}$ be the central reduction of $L_c(\tau;R)\otimes_Rk$. Recall that $e=\sum_{\gamma\in\Gamma_n}\gamma$. Let the simple module over $H_c(\Gamma)_\C$ labeled by the irreducible $\Gamma$-representation $\tau$ be denoted by $^fL_c(\tau;\CC)$. Similar to the simple modules over the spherical Cherednik algebra, we have $^f\overline{L_c(\tau;R)_k}$. Note that $L_c(\tau;\CC)=\null^fL_c(\tau;\CC)^\Gamma$ and $\overline{L_c(\tau;R)_k}=\null^f\overline{L_c(\tau;R)_k}^\Gamma$.

\begin{lemma}\label{lem: lift_sph}\footnote{The author is grateful to Roman Bezrukavnikov for access to his unpublished work where the author learned  this argument.}
For any parameter $c$ and any $\tau\in\hbox{Irrep}(\Gamma)$, we have $^fL_c(\tau;\CC)^{\Gamma}=0$ if and only if $^f\overline{L_c(\tau;R)_k}^{\Gamma}=0$.
\end{lemma}
\begin{proof}
If $^fL_c(\tau;\CC)^{\Gamma}=0$ then we can choose $^fL_c(\tau;R)$ so that  $^fL_c(\tau;R)^{\Gamma}=0$. Therefore, clearly we have $^f\overline{L_c(\tau;R)_k}^{\Gamma}=0$.

Conversely, for any weight space $^fL_c(\tau;\CC)[\alpha]$, for $p>>0$ we have an isomorphism $^fL_c(\tau;\CC)[\alpha]\to \null^f\overline{L_c(\tau;R)_k}[\alpha]$. If $^fL_c(\tau;\CC)^{\Gamma}\neq0$, then there is some weight $\alpha$ such that $^fL_c(\tau;\CC)[\alpha]^{\Gamma}\neq0$, and therefore $^f\overline{L_c(\tau;R)_k}[\alpha]^{\Gamma}\neq0$.
\end{proof}

Recall that in the terminology of \cite{BE}, such representations is said to be asperical. The asperical locus in $H^2(\Hilb;\QQ)$ (defined to be the locus where $^sH_c(\Gamma)_\CC$ has infinite global dimension) consists of values $c$ such that $H_c(\Gamma)$ has an asperical module.

For any  aspherical value $c$, define a filtration on $\sO_c$ by Serre subcategories \[\sO_c^{\leq d}:=\angl{L_c(\tau;\C)\mid \codim\supp L_c(\tau;\C)\leq d}.\]
Also we have a filtration on $\Mod_0\null ^{s}H_c(\Gamma_2)_k$ by Serre subcategories \[\Mod_0\null ^{s}H_c(\Gamma_2)_k^{\leq d}:=\angl{L_c(\tau;p)\mid \deg(Z_\tau)\leq d}.\]
These two filtrations are compatible in the following sense.

Let $A$ and $A'$ be two alcoves sharing a wall $H$. Assume $c$ is in alcove $A$, and $c'$ in $A'$. Assume moreover that $c_0$ is on $H$ but not any other walls.

\begin{prop}\label{thm: lift_char0}Let $\Gamma_1=\ZZ_1$. Suppose the codimension of support of $L_c(\tau;\bbC)$ is $d$. Then $\overline{L_c(\tau;R)_k}$ is a nonzero object in \[\Mod \null ^{s}H_c(\Gamma_2)_k^{\leq d}/\Mod\null ^{s}H_c(\Gamma_2)_k^{\leq d+1}.\]
\end{prop}

The following Lemma, which is the only place where we use the condition $n=2$ and $\Gamma_1=\ZZ/l\ZZ$, is checked by explicit description of the central charge polynomials in Section~\ref{sec: central_charge}.

\begin{lemma}\label{lem:char0-charp-max}
Let $\Gamma_1=\ZZ/l\ZZ$ and $n=2$. Let $H$ be a codimension-1 wall on which there is some $\theta\in\hbox{Irrep}(\Gamma_2)$  with $Z_{L_c(\theta;k)}(\nu)$ vanishes of degree 2. Then $L_c(\theta;\CC)$ is a finite dimensional representation of $^sH_c(\Gamma_2)$, and $\theta$ is the only irreducible representation of $\Gamma_2$ such that $Z_{L_c(\theta;k)}(\nu)$ vanishes on $H$.
\end{lemma}

If $L_c(\theta,\CC)$ is a finite dimensional irreducible representation of $^sH_c(\Gamma_2)$ with $T_{c\to c_0}(L_c(\theta;\CC))=0$, then for $k$ with large enough characteristic, $L_c(\theta;R)_k$ supported on $0\in \Aff^{4(1)}/{\Gamma_2}$. Therefore we have $L_c(\theta;R)_k\cong \overline{L_c(\theta;R)_k}$, and $\overline{L_c(\theta;R)_k}$ is an irreducible representation  of the same dimension as $\dim_{\CC}L_c(\theta;\CC)$.
By definition of the central charge polynomial, $Z_\theta(\nu)$ vanishes of degree 2 on $H$. In particular, taking into account of Lemma~\ref{lem:char0-charp-max}, we have the following lemma.

\begin{lemma}\label{lem:char0-charp-max2}
The  irreducible representation  $L_c(\theta,\CC)$ of $^sH_c(\Gamma_2)$ is a finite dimensional with $T_{c\to c_0}(L_c(\theta;\CC))=0$ for any $c_0\in H$ if and only if $Z_\theta(\nu)$ vanishes of degree 2 on $H$.\end{lemma}

Now we are ready to prove Theorem~\ref{thm: main} and Proposition~\ref{thm: lift_char0} in the case when $n=2$. Hopefully part of the proof will generalize to a more general set-up.

There are two cases.

Case 1:
In category $\sO_c$ there is a finite dimensional irreducible object $L_c(\theta;\CC)$ for some $\theta\in\hbox{Irrep}(\Gamma_2)$ such that $T_{c\to c_0}(L_c(\theta;\CC))=0$. In this case, by Lemma~\ref{lem: lift_sph}, the only $\tau \in \hbox{Irrep}(\Gamma_2)$  such that $Z_\tau(\nu)$ that vanishes on $H$ is $\tau=\theta$. This in turn forces $Z_\theta(\nu)$ to have vanishing order 2 on the wall $H$. This proves Proposition~\ref{thm: lift_char0} in this case.

In characteristic zero, $T_{c\to c\rq{}}(L_c(\theta;\CC))$ in concentrated in degree 2 as complex of $^sH_{c\rq{}}(\Gamma_2)$-modules. Therefore, over a field $k$ with characteristic $p>>0$, the complex $T_{c\to c\rq{}}(\overline{L_c(\theta;R)_k})$ has non-trivial cohomology in degree 2, and all cohomologies in degree more than 2. Moreover, on the quotient $\Mod_0\null^sH_c(\Gamma_2)_k/\angl{\overline{L_c(\theta;R)_k}}$ the functor $T_{c\to c\rq{}}$ induces a Morita equivalence, which is fits into a commutative diagram \[\xymatrix{\Mod_0\null^sH_c(\Gamma_2)_k/\angl{\overline{L_c(\theta;R)_k}}\ar[dr]_{T_{c\to c_0}}\ar[rr]^{T_{c\to c\rq{}}}&&\Mod_0\null^sH_{c\rq{}}(\Gamma_2)_k/\angl{\overline{L_{c\rq{}}(\theta;R)_k}}\ar[dl]^{T_{c\rq{}\to c_0}}\\
&\Mod_0\null^sH_{c_0}(\Gamma_2)_k&}.\] So, Theorem~\ref{thm: main} is true in this case.

Case 2:
For any $\tau\in\hbox{Irrep}(\Gamma_2)$ with $T_{c\to c_0}(L_c(\tau;\CC))=0$, the corresponding irreducible object $L_c(\theta;\CC)$ is infinite dimensional. In this case, none of such $\tau$ can have $L_c(\tau;\CC)$ being finite dimensional. This means, for any such $\tau$, the codimension of support of $L_c(\tau;\CC)$ has to be 1. This in turn, by Lemma~\ref{lem:char0-charp-max2}, implies that $Z_\tau(\nu)$ vanishes on $H$ with order 1. Then Lemma~\ref{lem: lift_sph} yields $\overline{L_{c_0}(\tau;R)_k}$ is aspherical. Therefore, $Z_{\overline{L_c(\tau;R)_k}}(\nu)$ vanishes on degree 1 on $H$. This proves Proposition~\ref{thm: lift_char0} in this case.

In order to finish the proof, it only remains to show that for such $\tau$, $T_{c\to c'}(\overline{L_c(\tau;R)_k})$ as a complex in $\Mod_0\null ^sH_{c'}$ is concentrated in degree 1.
Note that $T_{c\to c'}(L_c{\tau;\CC})$ has homological degree no more than 1, therefore so is $T_{c\to c'}(\overline{L_c(\tau;R)_k})$.
Then similar to the previous case, the commutativity of $T_{c\to c_0}=T_{c'\to c_0}\circ T_{c\to c'}$ implies that in homological degree zero $T_{c\to c'}(\overline{L_c(\tau;R)_k})$ vanishes. Also, $T_{c\to c'}$ induces a Morita equivalence when passing to $\Mod_0\null^sH_c(\Gamma_2)_k/\Mod_0\null^sH_c(\Gamma_2)_k^{\leq 1}$. This finishes the proof.

\section{Dimensions of irreducible objects}
Recall that $K_0(\Mod\null  ^{s}\calH_c(\Gamma_n))\cong K_0(X)\cong K_0(\Gamma_n)$, and the irreducible objects in $\Mod_0\null  ^{s}\calH_c$ are labeled by the irreducible representations of $\Gamma$. Recall that for an irreducible representation $\tau$ of $\Gamma$, the corresponding irreducible object in $\Mod_0\null  ^{s}\calH_c(\Gamma_n)$ will be denoted by $L_c(\tau;p)$. As has been seen in Lemma~\ref{lem: dim_poly}, $\dim_k(L_c(\tau;p))$ is a polynomial in $c$, as long as $c$ varies in an alcove in the affine hyperplane arrangement.

\begin{problem}\label{prob: poincare}
Assume $p$ is large enough, compute the graded characters of the irreducible representation $L_c(\tau;p)$.
\end{problem}

A weaker version of this Problem is: compute the Poincar\'{e} polynomial of the irreducible object $L_c(\tau;p)$ for regular parameter $c$.

Note that the Poincar\'{e} polynomial specializes to the dimension polynomial $\dim_k(L_c(\tau;p))$.
When the parameter $c$ lies in the alcove containing $0$, the  irreducible modules $L_c(\tau;p)$ are quotients of the Verma modules, which are the $\tau$-isotypical components in the space of \textsl{$c$-quasi-invariant polynomials}. The Poincar\'e polynomials of the Verma modules have been calculated by Berest, Chalykh, Felder, and Veselov in \cite{BC} and \cite{FV}.

The ring $K:=\oplus_{n\geq 0}K_0(\Gamma_n)$, endowed with the parabolic induction and restriction functors of finite group representations, is a Hopf algebra. The existence of parabolic induction and restriction functors in \cite{BE} show that in order to know the Poincar\'{e} polynomials of any $L_c(\tau;p)$, it suffices to calculate the Poincar\'{e} polynomials of the irreducible $\Gamma_n$-representations which are algebraic generators.

For $\Gamma_1=\bbZ/l\bbZ$ and $c$ lying in the alcove containing $0$, for a particular set of irreducible modules which generates $K$ multiplicatively, we construct the resolutions of those irreducible modules by Verma modules in this section. As a consequence, for such $\tau$, the Poincar\'{e} polynomials of $L_c(\tau;p)$ will be obtained.

\subsection{Trivial representation of $\mathfrak{S}_n$}

We work over a field of characteristic $p>0$.
Let $\mathfrak{h}$ be the reflection representation of $\mathfrak{S}_n$.
Let $m$ be an integer and let $Q_m(\mathfrak{h}):=Q_m$ be the $m$-quasi-invariants on $\mathfrak{h}^*$. Let $\widetilde{Q_m}$ be the quasi-invariants on the Frobenius neighborhood of the origin. 
The space $\widetilde{Q_m}$ carries actions of $\fS_n$ and $eH_me$ which satisfies the Schur-Weyl duality. 
In other words, the multiplicity space on $\widetilde{Q_m}$ corresponds to irreducible representations of $\mathfrak{S}_n$ gives the irreducible representations of the spherical rational Cherednik algebra.

We want to calculate the Poincar\'e series of the isotypical component on $\widetilde{Q_m}$ corresponds to the trivial representation. A resolution of $\widetilde{Q_m}$ is given by the Koszul complex
\[\cdots\to Q_m\otimes \wedge^2\mathfrak{h}^{(1)}\to Q_m\otimes \mathfrak{h}^{(1)}\to Q_m.\]
Note that here $\mathfrak{h}^{(1)}$ has degree $p$.
We decompose $Q_m$ according to the $\mathfrak{S}_n$-$eH_me$ bimodule,
\[ Q_m=\oplus_{\tau\in\widehat{\mathfrak{S}_n}}\tau^*\otimes M_m(\tau)e.\]
Then \cite{BEG} and \cite{FV} give the Poincar\'e series of $M_m(\tau)e$.
\[P_t(M_m(\tau)e)=t^{\xi_m(\tau)}{{K_\tau(t)}\over{\prod_{i=2}^n(1-t^i)}},\]
where ${K_\tau(t)}$ is the Poincar\'e series of the $\tau$-component in the spaces of harmonic polynomials, and $\xi_m(\tau)$ is the integer by which the element $\sum_{s\hbox{ a reflection in }\mathfrak{S}_n}(1-s)$ acts on $\tau$.

It is a classical formula that \[K_\tau(t)=(\prod_{k=1}^n(1-t^k))(\prod_{(i,j)\in\tau}{{t^{l(i,j)}}\over{1-t^{h(i,j)}}}),\] where $l(i,j)$ is the leg length of the box $(i,j)$ in the partition $\tau$, and $h(i,j)$ is the hook length.

So the Poincar\'e series of the irreducible representation of $eH_me$ corresponding to the trivial representation of $\mathfrak{S}_n$ is \[\sum_{s=1}^n(-1)^{s-1}P_t(Q_m\otimes\wedge^{s-1}\mathfrak{h}^{(1)})^{\mathfrak{S}_n}.\]
This is the same as
\begin{eqnarray*}
&&\sum_{s=1}^n(-1)^{s-1}P_t(Q_m\otimes\wedge^{s-1}\mathfrak{h}^{(1)})^{\mathfrak{S}_n}\\
&=&\sum_{s=1}^n(-1)^{s-1}t^{(s-1)(mn+p)+\binom{s}{2}}{{1}\over{\prod_{i=1}^{s-1}(1-t^i)\prod_{i=1}^{n-s}(1-t^i)}}{{1-t}\over{1-t^n}}\\
&=&{{1-t}\over{1-t^n}}\sum_{s=0}^{n-1}{{(-1)^st^{s(mn+p)+\binom{s+1}{2}}}\over{\prod_{i=1}^{s}(1-t^i)\prod_{i=1}^{n-1-s}(1-t^i)}}.
\end{eqnarray*}
Now let us briefly recall an identity proved in \cite{KC}. Define $[n]:=\frac{t^n-1}{t-1}$, $[n!]:=[n][n-1]\cdots[1]$, and $(x+a)^n_t:=(x+a)(x+ta)(x+t^2a)\cdots(x+t^{n-1}a)$. Then, we have the following identity \[(x+a)_t^n=\sum_{j=0}^n{\frac{[n]!}{[j]![n-j]!}}t^{\binom{j}{2}}a^jx^{n-j}.\]
Using notations in \cite{KC}, \[{{1-t}\over{1-t^n}}\sum_{s=0}^{n-1}{{(-1)^st^{s(mn+p)+\binom{s+1}{2}}}\over{\prod_{i=1}^{s}(1-t^i)\prod_{i=1}^{n-1-s}(1-t^i)}}={\frac{1-t}{\prod_{i=1}^n(1-t^i)}}\sum_{s=0}^{n-1}{\frac{[n-1]!}{[s]![n-1-s]!}}t^{\binom{s+1}{2}}(-t^{mn+p})^s.\]
According to the identity from \cite{KC} recalled above, it is equal to \[ {\frac{1-t}{\prod_{i=1}^n(1-t^i)}} (1+(-t^{nm+p}))^{n-1}_t={{1-t}\over{1-t^n}}\prod_{i=1}^{n-1}{{1-t^{mn+p+i}}\over{1-t^i}}.\]

To summarize, we have the following Lemma.
\begin{lemma}
The Poincar\'e series of the irreducible representation of $eH_me$ corresponding to the trivial representation of $\mathfrak{S}_n$ is
\[{{1-t}\over{1-t^n}}\prod_{i=1}^{n-1}{{1-t^{mn+p+i}}\over{1-t^i}}.\]
\end{lemma}

\subsection{Characters of the wreath product $\Gamma_n=(\ZZ/l)^n\rtimes\mathfrak{S}_n$}
Recall that the irreducible representations of $\Gamma_n$ are in one-to-one correspondence with $l$-partitions of $n$, i.e., $\boldsymbol{\lambda}=(\lambda^1,\cdots,\lambda^l)$ where $\lambda^i$'s are partitions such that $\sum_{i=1}^l|\lambda^i|=n$.
More explicitly, for an $l$-partition $\boldsymbol{\lambda}=(\lambda^1,\cdots,\lambda^l)$, let $l_r$ be the numbder of rows in $\lambda^r$, and let $I_\lambda(r)=\{\sum_{i=1}^{r-1}|\lambda^i|+1,\sum_{i=1}^{r-1}|\lambda^i|+2,\cdots,\sum_{i=1}^{r}|\lambda^i|\}$. Let $\mathfrak{S}_\lambda=\mathfrak{S}_{I_\lambda(1)}\times\cdots\mathfrak{S}_{I_\lambda(l)}$. Then, the $l$-partition $\boldsymbol{\lambda}$ corresponds to the irreducible representation of $\Gamma_n$ constructed as $\Ind_{(\ZZ/l\ZZ)^{\{1,2,\dots,n\}}\rtimes\mathfrak{S}_\lambda}^{\Gamma_n}(\phi^1\cdot\lambda^1\otimes\cdots\otimes\phi^l\cdot\lambda^l)$, where $\phi^r$ is the character $\det^r$ of $(\ZZ/l\ZZ)^{I_\lambda(r)}$. (We follow the convention in \cite{GL}.)

In this section let $\mathfrak{h}\cong\CC^n$ be the reflection representation of $\Gamma_n$.
Let $m$ be an integer valued function on the set of reflections in $\Gamma_n$, constant on congugacy classes, and let $Q_m(\mathfrak{h}) := Q_m$ be the $m$-quasi-invariants on $\mathfrak{h}$. Let $\widetilde{Q_m}$
be the quasi-invariants on the Frobenius neighborhood of the origin. Again, the multiplicity space $\widetilde{Q_m}(\tau)$ on $\widetilde{Q_m}$ corresponds
to irreducible representations $\tau$ of $\Gamma_n$ gives the irreducible representations of the spherical
rational Cherednik algebra.
Section~8.2 of \cite{BC} gives the Poincar\'{e} series of $Q_m(\tau')e$ as \[P_t(Q_m(\tau)e)=t^{\xi_m(\tau')}\cdot P_t((k[\mathfrak{h}]\otimes (\tau')^*)\Gamma_n),\]
where $\tau':=kz_{-k}(\tau)$ is the $kz$-twist defined in \cite{O}.

Let $\tau(i)$ be the $l$-partition $\boldsymbol{\lambda}$ whose $\lambda^i=(n)$ and $\lambda^j$ is empty for all $j\neq i$. We want to calculate the Poincar\'e series of the isotypical component on $\widetilde{Q_m}$ corresponds to the $\tau(i)$. A resolution of $\widetilde{Q_m}$ is given by the Koszul complex
\[\cdots\to Q_m\otimes \wedge^2\mathfrak{h}^{(1)}\to Q_m\otimes \mathfrak{h}^{(1)}\to Q_m.\]

The Poincar\'e series of the irreducible representation of $eH_me$ corresponding to the representation $\tau(i)$ of $Gamma_n$ is \[\sum_{s=0}^n(-1)^sP_t((Q_m\otimes\wedge^s\mathfrak{h}^{(1)}\otimes\tau(i)^*)^{\Gamma_n}).\]
The representation $\wedge^s\mathfrak{h}$ corresponds to the $l$-partition $((1)^s,\emptyset,\cdots,\emptyset,(n-s))$.
Hence, $\wedge^s\mathfrak{h}\otimes\tau(i)^*$ corresponds to the $l$-partition $\boldsymbol{\lambda}$ whose $i$-th component is $(n-s)$ and $i+1$-th component is $(1)^s$. By the adjoint of $\Ind$ and $\Res$,  we have the following sequence of isomorphisms of graded modules $\Hom_{\Gamma_n}(\boldsymbol{\lambda}, k[\mathfrak{h}])\cong\Hom_{(\ZZ/l\ZZ)_s\times(\ZZ/l\ZZ)_{n-s}}(\phi^i\cdot(n-s)\otimes\phi^{i+1}\cdot((1)^s),\Res k[\mathfrak{h}])\cong\Hom_{\mathfrak{S}_{n-s}}((n-s), \otimes_{q=1}^{n-s}k[x_q^l]x_q^{i} )\otimes\Hom_{\mathfrak{S}_s}(((1)^s),\otimes_{q=1}^sk[x_q^l]x_q^{i+1})$.
Now, each individual Poincar\'e polynomial can be calculated using the hook-length formula. Note that on this $\boldsymbol{\lambda}$, the value $\xi_m$ is $s(nm_0+lm_{i+1})$.

\begin{eqnarray*}
&&\sum_{s=0}^n(-1)^sP_t((Q_m\otimes\wedge^s\mathfrak{h}^{(1)}\otimes\tau(i)^*)^{\Gamma_n})\\
&=&\sum_{s=0}^n(-1)^st^{\xi_m(\boldsymbol{\lambda})+sp}t^{s(i+1)}\prod_{k=1}^s{\frac{t^{l(k-1)}}{1-t^{lk}}}t^{(n-s)i}\prod_{k=1}^{n-s}{\frac{1}{1-t^{lk}}}\\
&=&{\frac{t^{ni}}{\prod_{k=1}^n(1-t^{kl})}}\sum_{s=0}^n(-t^{m_0n+p+1+lm_{i+1}})^st^{l\binom{s}{2}}{\frac{\prod_{k=1}^n(1-t^{kl})}{\prod_{k=1}^s(1-t^{kl})\prod_{k=1}^{n-s}(1-t^{kl})}}\\
&=&{\frac{t^{ni}\prod_{k=0}^{n-1}(1-t^{lk+m_0n+p+1+lm_{i+1}})}{\prod_{k=1}^n(1-t^{kl})}}.
\end{eqnarray*}

Summarizing the calculation above, we have the following Proposition.
\begin{prop}
Let $\tau(i)$ be the $l$-partition $\boldsymbol{\lambda}$ whose $\lambda^i=(n)$ and $\lambda^j$ is empty for all $j\neq i$. The Poincar\'e series of the irreducible representation of $eH_me$ corresponding to $\tau(i)$ of $Gamma_n$ is
\[{\frac{t^{ni}\prod_{k=0}^{n-1}(1-t^{lk+m_0n+p+1+lm_{i+1}})}{\prod_{k=1}^n(1-t^{kl})}}.\]
\end{prop}

\section{The Chern character map of the resolution}
The central charge map $Z:H^2(\Hilb;\bbQ)\to \Hom_{\ZZ}(K_0(\Hilb),\bbQ)$, which is defined by modifying the dimension polynomials of the irreducible modules over $^sH_c$,  is related to the Chern character map $\ch: K_0(\Hilb)_\bbQ\to H^*(\Hilb;\bbQ)$, as will be explained in more details in this section.
Thanks to the work of Ginzburg and Kaledin, the multiplicative structure of $H^*(\Hilb,\QQ)$ is easily described. The abelian group structure of $K_0(\Hilb)$ is given by the symplectic McKay correspondence.
It is a long-standing question to calculate the Chern character map in terms of the natural bases of the two sides.

\subsection{The central charge and the Chern character map}

Let $k$ be a seperably closed field of characteristic $p>0$. Let $\Gamma$ be an arbitrary symplectic reflection group acting on $\Aff^{2n}_k$ by symplectic reflections. Let $\bfX$ be a symplectic resolution of $\Aff^{2n}/\Gamma$. According to \cite{BK04}, there is a derived equivalence $$D^b(\Coh\bfX)\cong D^b(\Mod W_n^\Gamma)$$ where $W_n$ is the ring of differential operators on $\Aff^n$. This derived equivalence is given by a vector bundle $\sE_0$ on $\bfX$. It is shown in \cite{BK04} that any of such a vector bundle $\sE_0$ lifts to characteristic zero. Therefore, for simplicity in wht follows in this section and the next one, we work over a field of characteristic zero. (This is not essential. One can replace $H^*(\bfX;\QQ)$ by $H^*(\bfX;\QQ_r)$ for $r\neq p$, and all the statements in this section are still true.)
The vector bundle $\calE_0$ in \cite{BK04} is not unique. The non-uniqueness has been studied by Losev in \cite{Los}, together with a preferred choice. Under the derived correspondence $D^b(\Coh\bfX)\cong D^b(\Mod W_n^\Gamma)$, there is a set of vector bundles $\{\sV_\alpha\}$ on the symplectic resolution $\bfX$ corresponding to the indecomposible projective modules over $W_n^\Gamma$, whish are in turn labeled by the irreducible representations of $\Gamma$. The classes of $\{\sV_\alpha\}$ in the Grothendiech group form a basis of $K_{\QQ}(\bfX)$.

On the other hand, the cohomology ring $H^*(\bfX,\bbQ)$ has an explicit description. Let $\bbQ[\Gamma]$ be the group ring. It is filtered by the codimension of the fixed point loci. This filtration induces a filtration on the center $Z\bbQ[\Gamma]$, whose associated graded ring will be denoted by $\gr Z\bbQ[W]$.
\begin{theorem}[\cite{EG} and \cite{GK}]\label{thm: Ginz_Kal}
The algebra $H^*(\bfX,\bbQ)$ is isomorphic to the algebra $\gr Z\bbQ[\Gamma_n]$.
\end{theorem}

The following problem is raised by Etingof, Ginzburg, and Kaledin, and is referred to as the Chern character problem.
\begin{problem}[\cite{EG} and \cite{GK}]\label{prob:chern}
Express explicitly the map \[K_0(\Gamma_n)\to \gr Z\bbQ[\Gamma_n]\] induced by the Chern character \[\ch:K_0(\bfX)\to H^*(\bfX;\bbQ).\]
\end{problem}

The character group of $\Gamma$ will be denoted by $\hat{\Gamma}$. We define polynomials on $H^2(\bfX;\QQ)$
$$\mathfrak{l}_{\sL_\alpha}(n_\tau)_{\tau \in \hat{\Gamma}}:=\chi(\sL_\alpha\otimes(\otimes_{\tau\in \hat{\Gamma}}\sV_\tau^{n_\tau}))$$ in the variables $n_\tau$.

The dimension of the irreducible representations over s field of positive characteristic can be calculated by modifying these polynomials. For an explicit illustration of this we refer to \S~\ref{sec: central_charge}.

\begin{prop}\label{prop: chern_dim}
For an arbitrary basis $\{b\}$ of $H^*(\bfX;\QQ)$ with the change of bases given by  $b=\sum_{\alpha\in \hbox{Irrep}\Gamma} h^b_\alpha\ch(\sV_\alpha)$, we have $\mathfrak{l}_{\sL_\alpha}(b)=h^b_\alpha$.
\end{prop}
\begin{proof}
To calculate the polynomial $$\mathfrak{l}_{\sL_\alpha}(n_\tau)_{\tau \in \hat{\Gamma}}:=\chi(\sL_\alpha\otimes(\otimes_{\tau\in \hat{\Gamma}}\sV_\tau^{n_\tau}))$$ in the variables $n_\tau$ from the Chern character map. In positive characteristic, the dimension of the irreducible representations can be calculated by modifying these polynomials, as will be done in later sections.

The group $\Gamma$ is generated by the classes of symplectic reflections, therefore, there is a basis of $H^*(\Hilb^2,\QQ)$ given by $\{\ch(\otimes_{\tau\in \hat{\Gamma}}\sV_\tau^{a_\tau})\mid (a_\tau)_\tau\in I\}$ for some set $I\subset \ZZ^{|\hat{\Gamma}|}$. Also $\{\otimes_{\tau\in \hat{\Gamma}}\sV_\tau^{a_\tau}\mid (a_\tau)_\tau\in I\}$ form a basis of $K_{\QQ}(\Hilb^2)$. For simplicity, we write $\sO(\underline{a})$ for the line bundle $\otimes_{\tau\in \hat{\Gamma}}\sV_\tau^{a_\tau}$. Let $$[\sV_\alpha]=\sum_{\underline{a}\in I}m_\alpha^{\underline{a}}[\sO(\underline{a})].$$
Therefor, $[\sV_\alpha^*]=\sum_{\underline{a}\in I}m_\alpha^{\underline{a}}[\sO(\underline{a})^*]$.

The polynomials $\mathfrak{l}_{\sL_\alpha}(n_\tau)_{\tau \in \hat{W}}$, considered as functions on $H^*(\Hilb^2,\QQ)$, are $\QQ$-linear functions, hence, $$\sum_{\underline{a}\in I}m_\beta^{\underline{a}}\mathfrak{l}_{\sL_\alpha}(\underline{a})=\delta_{\alpha,\beta}.$$
In other words, the value of the linear function $\mathfrak{l}_{\sL_\alpha}(n_\tau)_{\tau \in \hat{W}}$ at the basis element $\ch(\sO(\underline{a}))$ is given by $m^{-1,\underline{a}}_\alpha$, where $m^{-1,\underline{a}}_\alpha$ is the $(\underline{a},\alpha)$-entry of the inverse matrix of $(m_\beta^{\underline{a}})$.
Therefore, for an arbitrary basis $\{b\}$ of $H^*(\bfX;\QQ)$ with $b=\sum h^b_\alpha\ch(\sV_\alpha)$, we have $\mathfrak{l}_{\sL_\alpha}(b)=h^b_\alpha$.
\end{proof}

In this section and the next one, we study the Chern character map in some special cases.

\subsection{The topology of the punctual Hilbert scheme}

From this section on we concentrate on the case when $\Gamma_1=\ZZ/l$ and $n=2$. In this section we describe the cohomology ring and the $K$-group of the symplectic resolution, and give a formula of the Chern character map.

We present $\Gamma_2=(\ZZ/l\ZZ)^2\rtimes \mathfrak{S}_2$ as $\langle\xi,\eta,\sigma\mid \xi^l,\eta^l, \sigma^2, \sigma\eta\sigma=\xi\rangle$. Now we look at the conjugacy  classes of elements in $\Gamma_2$.
There are $l$ conjugacy classes in $\Gamma_2$ consists of symplectic reflections. They are represented by $\xi^i$ with $i=1,\cdots,l-1$, and $\sigma$. There are ${{l}\choose{2}}+2(l-1)$ conjugacy classes whose fixed point loci consist of the origin only. They are represented by $\sigma\xi^i$ with $i=1,\cdots,l-1$, $\xi^i\eta^i$ with $i=1,\cdots,l-1$, and $\xi^i\eta^j$ with $i\neq j$.

If we write $[g]$ for $\sum_{h\sim g}h\in \QQ[\Gamma_2]$, the natural basis of $\gr Z\QQ[\Gamma_2]$ is given by $\{[g]\mid g\in W\}$. They satisfies $[\xi^i]\cdot[\xi^j]=[\xi^i\eta^j]$, $[\xi^i]^2=2[\xi^i\eta^i]$, and $[\sigma]\cdot[\xi^i]=2[\sigma\xi^i]$.

Let $\widetilde{\Aff^2/\ZZ_l}\to \Aff^2/\ZZ_l$ be the minimal resolution of Kleinian singularity. Then a symplectic resolution of $\Aff^4/W$ is given by $\Hilb^2=\Hilb^2(\widetilde{\Aff^2/\ZZ_l})$. It fits in the basic diagram
\begin{equation}\label{eqn:diag_res}\begin{xymatrix}{
\Bl_\triangle(\widetilde{\Aff^2/\ZZ_l}\times\widetilde{\Aff^2/\ZZ_l})\ar[r]^{q}\ar[d]^{p}&\widetilde{\Aff^2/\ZZ_l}\times\widetilde{\Aff^2/\ZZ_l}\ar[d]\\
\Hilb^2(\widetilde{\Aff^2/\ZZ_l}\times\widetilde{\Aff^2/\ZZ_l})\ar[r]&\widetilde{\Aff^2/\ZZ_l}\times\widetilde{\Aff^2/\ZZ_l}/\mathfrak{S}_2.
}
\end{xymatrix}\end{equation}

Let $C\subset \widetilde{\Aff^2/\ZZ_l}$ be the exceptional divisor. Recall that $C$ is a chain of $\PP^1$'s, each having self-intersection number -2. We number them as $C_1,\cdots,C_{l-1}$ such that $[C_i][C_{i+1}]=1$ and $[C_i][C_j]=0$ if $|i-j|>1$.

We now describe the cohomology ring of $\Hilb^2$. Since $\Hilb^2$ deformation retracts to the punctual Hilbert scheme $X=\Hilb^2_C(\widetilde{\Aff^2/\ZZ_l})$, we concentrate on the latter. The scheme $X$ has ${{l}\choose{2}}+2(l-1)$ irreducible components, coming from the strict transform of $C_i\times C_j\subset\widetilde{\Aff^2/\ZZ_l}\times\widetilde{\Aff^2/\ZZ_l}$ into $\Hilb^2$ under the maps $p$ and $q$. We now describe these components.

For each irreducible component $C_i\subset C$, there are two irreducible components coming out of the strict transform of $C_i\times C_i$. One component is isomorphic to $\PP^2$ which we will denote by $\PP^2_i$. The other component is isomorphic to the rational ruled surface $\PP(\sO_{\PP^1}\oplus\sO_{\PP^1}(-4))\cong\PP(\sO(2)\oplus\sO(-2))$ which will be denoted by $S_i$. We identify $\sO_{\PP^1}(2)\oplus\sO_{\PP^1}(-2)$ with $\hbox{T}(\hbox{T}^*\PP^1)|_{\PP^1}$ where $\PP^1\subset \hbox{T}^*\PP^1$ is the zero section. The fiber of $S_i\to\PP^1$ over $x\in \PP^1$ is $\PP(\hbox{T}_x\hbox{T}^*\PP^1)$. These two components $\PP^2_i$ and $S_i$ are glued together along a common divisor $\PP^1$. This $\PP^1$ sits inside $\PP^2_i$ as a degree 2 irreducible hypersurface.
In $S_i$ this divisor $\PP^1$ is embedded as a section of this rational ruled surface which corresponds to the subboundle $\sO(2)\subseteq\sO(2)\oplus\sO(-2)$.(Over each point $x$ in the zero section of $\hbox{T}^*\PP^1$, this $\PP^1$ corresponds to the direction of $\hbox{T}_x\PP^1$ in $\hbox{T}_x\hbox{T}^*\PP^1$.)
In fact, the normal bundle $N_{C_i}\widetilde{\Aff^2/\ZZ_l}\cong\sO(C_i)|_{C_i}\cong\sO_{\PP^1}(-2)\cong \hbox{T}^*\PP^1$, and also $N_{\triangle}(\widetilde{\Aff^2/\ZZ_l})^2|_{C_i}\cong N_{\triangle}(N_{C_i}\widetilde{\Aff^2/\ZZ_l})^2$. Therefore, the strict transform of $C_i\times C_i\subset\widetilde{\Aff^2/\ZZ_l}\times\widetilde{\Aff^2/\ZZ_l}$ into $\Hilb^2$ is isomorphic to the strict transform of $\PP^1\times\PP^1\subset \hbox{T}^*\PP^1\times \hbox{T}^*\PP^1$ into $\Bl_\triangle(\hbox{T}^*\PP^1\times \hbox{T}^*\PP^1)/\ZZ_2$, which is obviously $\PP^2\sqcup_{\PP^1}S$ as described above. The common $\PP^1$ in $S_i$ and $\PP^2$ is the strict transform of the diagonal.

For $C_i\neq C_j$, there is an irreducible component of $X$ coming from the strict transform of $C_i\times C_j$, which will be called $ P_{ij}$. We have $\PP_{ij}\cong\PP^1_i\times\PP^1_j$ if $[C_i][C_j]=0$ in $\widetilde{\Aff^2/\ZZ_l}$, and $ P_{ij}\cong\Bl_{*}\PP^1_i\times\PP^1_j$ if $[C_i][C_j]=[*]$ in $\widetilde{\Aff^2/\ZZ_l}$ where $*$ is a point.

Let us write down a basis of the cohomology rings of each of the irreducible components. We take the canonical basis of $H^2(\PP^2_i)$  as $q_i$, and $q_i^2=p_i$. We denote the Poincar\'{e} dual of the zero section of $S_i\cong\PP(\sO_{\PP^1}\oplus\sO_{\PP^1}(-4))$ in $H^2(S_i)$ by $c_i$, the Poincar\'{e} dual of the fiber by $f_i$, and the fundamental class in $H^4(S_i)$ by $s_i$. (Here we follow the convention in Hartshorne and therefore $c_0^2=-degree$.) If $i=j-1$, we denote the Poincar\'{e} dual of the exceptional divisor in $H^2(\Bl_{*}\PP^1_i\times\PP^1_j)$ by $e_i$. No matter whether $i$ and $j$ are adjacent or not, we denote the Poincar\'{e} dual of  $[*\times\PP^1_j]$  by $l_{j,i}$, and the Poincar\'{e} dual of $[\PP^1_i\times *]$ by $l_{i,j}$. The fundamental class will be denoted by $p_{i,j}$.

Besides the basis of $H^*(\Hilb^2,\QQ)$ coming from the natural basis of $\gr^{\hbox{top}}Z\QQ[W]$ described at the beginning of this section, there is another basis of $H^*(\Hilb^2,\QQ)$ coming from the topology of $X$ which will be described here. The basis of $H^2(\Hilb^2,\QQ)$ comes from the divisor classes, which in turn corresponds to conjugacy classes of symplectic reflections in $W$. The basis of $H^4(\Hilb^2,\QQ)$ comes from irreducible component of $X$.

Note that in our case (and many other cases), the  resolution $\Hilb^2$ can be constructed as a Nakajima quiver variety (see, e.g., \cite{Kuz01}).
The basis of $H^4(\Hilb^2,\QQ)$ (resp. $H^{\hbox{mid}}$) coming from irreducible components coincide with the basis given by Nakajima in \cite{Nak94}. It is a natural question to ask what the matrix of transform is between this basis and the one coming from conjugacy classes in $\gr^{\hbox{top}}Z\QQ[\Gamma_2]$. In the case concerned in this paper, we will solve this problem by working out the multiplicative structure of $H^*(\Hilb^2,\QQ)$ under the topological basis.

Now we can describe the basis of $H^2(X)$ coming from symplectic reflections more explicitly. The divisor class coming from the symplectic reflection $\sigma$ is $d_0=\sum q_j+\sum c_j+2\sum f_j+\sum_{j=1}^{l-2} e_j$. The divisor coming from the symplectic reflection $\xi^i$ for $i=1,\cdots,l-1$ is $d_i=q_i+2f_i+\sum_{j\neq i} l_{i, j}$. The non-trivial multiplications of them are given by
\[d_0^2=\sum p_j-\sum p_{j,j+1};\]
\[d_0\cdot d_i=p_i+2s_i;\]
\[d_i^2=p_i;\]
\[d_i\cdot d_j=p_{ij}.\]

Now we study the $K$-theory of $\Hilb^2$.

Fix a primitive $l$-th root of unity $\omega$, the irreducible representations of $W$ can be written as:
\begin{itemize}
  \item the trivial representation, whose corresponding vector bundle on $\Hilb^2$ under the McKay correspondence is denoted by $\sV_0$;
  \item the 1 dimensional representation $k_i$ acted trivially by $\mathfrak{S}_2$ and via $\omega^i$ by $\ZZ_l$,  whose corresponding vector bundle is $\sV_{i}$;
  \item the sign representation of $\mathfrak{S}_2$ tensor with $k_i$, whose corresponding vector bundle is  $\sV_{\sigma,i}$;
  \item the irreducible 2 dimensional representation acted via $\left(\begin{smallmatrix}\omega^i&0\\0&\omega^j\end{smallmatrix}\right)$ by $\ZZ_l$,  whose corresponding vector bundle is $\sV_{i,j}$.
\end{itemize}

The main result of this section is the following proposition.
\begin{prop}\label{prop: chern char}
We have
\begin{eqnarray*}
\ch(\sV_0)&=&1;\\
\ch(\sV_i)&=&1+d_i+p_i/2;\\
\ch(\sV_\sigma)&=&1+d_0+\sum p_j/2-\sum p_{j,j+1}/2;\\
\ch(\sV_{\sigma,i})&=&1+d_i+d_0+\sum p_j/2-\sum p_{j,j+1}/2+3p_i/2+2s_i;\\
\ch(\sV_{0,i})&=&2+d_0+d_i+\sum p_j/2-\sum p_{j,j+1}/2+p_i/2+s_i;\\
\ch(\sV_{i,j})&=&2+d_j+d_i+d_0+\sum p_k/2-\sum p_{k,k+1}/2+p_{i,j}+p_i/2+p_j/2+s_i+s_j.
\end{eqnarray*}
\end{prop}

The rest of this section is devoted to the proof of this proposition.

\begin{lemma}\label{lem: k-grp eq}
In $K(\Hilb^2)$, we have $\sV_{i,j}=\sV_i\otimes(\sO\oplus\sV_{\sigma}\oplus\sV_{0,j}-\sV_{0,i})$.
\end{lemma}
\begin{proof}
We have, on the one hand, in $K^{\mathfrak{S}_2}(\widetilde{\Aff^2/\ZZ_l})$,
\begin{eqnarray*}
Rq_*p^*\sV_{i,j}&=&\sO_i\boxtimes\sO_j\oplus\sO_j\boxtimes\sO_i\\
&=&\sO_i\boxtimes\sO_i\otimes(\sO\oplus\sO_\sigma\oplus(\sO\boxtimes\sO_j\oplus\sO_j\boxtimes\sO)-(\sO\boxtimes\sO_i\oplus\sO_i\boxtimes\sO)),
\end{eqnarray*}
where $\sO_\sigma$ is endowed with the sign representation of $\mathfrak{S}_2$.
On the other hand,
\begin{eqnarray*}\label{eq: divisor-Procesi}
\sO_i\boxtimes\sO_j\oplus\sO_j\boxtimes\sO_i&=&\sO_i\boxtimes\sO_i\otimes(\sO\oplus\sO_\sigma\oplus(\sO\boxtimes\sO_j\oplus\sO_j\boxtimes\sO)-(\sO\boxtimes\sO_i\oplus\sO_i\boxtimes\sO))\\
&=&\sO_i\boxtimes\sO_i\otimes(Rq_*p^*(\sO\oplus\sV_{\sigma}\oplus\sV_{0,j}-\sV_{0,i}))\\
&=&Rq_*(q^*(\sO_i\boxtimes\sO_i)\otimes p^*(\sO\oplus\sV_{\sigma}\oplus\sV_{0,j}-\sV_{0,i}))\\
&=&Rq_*p^*(\sV_i\otimes(\sO\oplus\sV_{\sigma}\oplus\sV_{0,j}-\sV_{0,i})).
\end{eqnarray*}
\end{proof}

\begin{example}
We take a concrete eaxmple, i.e., the case when $\Gamma_2=B_2$.
We present $B_2$ as $\langle s_1,s_2,\sigma \mid s_1^2,\ s_2^2,\ \sigma^2,\ \sigma s_1\sigma=s_2\rangle$.
\par
There are 5 irreducible representations of $B_2$: the trivial representation $V_0$, $V_1=k$ with $\sigma=(-1)$, $V_2=k$ with $s_1=s_2=(-1)$, $V_3=k$ with $\sigma=s_1=s_2=(-1)$, and $V_4=\hh$. Let $\sV_i$ be the vector bundle on the symplectic resolution corresponding to the projective object $V_i\times \Aff^4$ under the McKay correspondence. Their corresponding simple object will be denoted by $\sL_i$.
\par
The symplectic resolution of $\Aff^4/B_2$ is given by the Hilbert scheme $\Hilb^2=\Hilb^2(\widetilde{\CC^2/\ZZ_2})$ where $\widetilde{\CC^2/\ZZ_2}\to \CC^2/\ZZ_2$ is the minimal resolution of Kleinian singularity. More concretely, $\widetilde{\CC^2/\ZZ_2}\cong \hbox{T}^*\PP^1$.
The Hilbert scheme $\Hilb^2(\widetilde{\CC^2/\ZZ_2})\cong\Bl_\Delta(\hbox{T}^*\PP^1\times \hbox{T}^*\PP^1)/\ZZ_2$.
The punctual Hilbert scheme $X=\Hilb^2_{\PP^1}(\hbox{T}^*\PP^1)\cong\PP^2\sqcup_{\PP^1}S$ where $S\cong\PP(\sO(2)\oplus\sO(-2))\cong\PP(\sO\oplus\sO(-4))$ with the base $\PP^1$ being the zero section of $\hbox{T}^*\PP^1$ and fiber of $x\in \PP^1$ being $\PP(\hbox{T}_x\hbox{T}^*\PP^1)$.
These two surfaces are glued together over a common divisor $\PP^1$ which are the proper transforms of the diagonal.
We write $\pi:\PP^2\sqcup S\to \Hilb^2$ as the natural map.
\par
The cohomology ring $H^*(X)$ of the central fiber, which is canonically isomorphic to the cohomology ring of the entire resolution, has a basis as follows.
We take the basis of $H^2(S)$ as $c_0$ and $f$, where $c_0$ is the Poincar\'{e} dual of the zero section, and $f$ is the Poincar\'{e} dual of the fiber.
We denote the canonical basis of $H^2(\PP^2)$ by $q$. Then the basis for $H^2(X)$ can be chosen as $d_1=q+2f$ and $d_2=q+c_0+2f$. (The common divisor $\PP^1$ is $2q$ in $H^2(\PP^2)$ and $c_0+4f$ in $H^2(S)$.)
We denote the fundamental class of $S$ by $s$ and the fundamental class of $\PP^2$ by $p$.
\par
There are 2 conjugacy classes of symplectic reflections in $B_2$, which give two divisors in $\Hilb^2$, i.e., $D_1$ corresponding to $s_i$ and $D_2$ corresponding to $\sigma$. We can restrict these two divisors to the central fiber and get $\pi_S^*[D_1]=[2F]$, $\pi_S^*[D_2]=[C_0+2F]$, $\pi_{\PP^2}^*[D_1]=[Q]$, and $\pi_{\PP^2}^*[D_2]=[Q]$.
\par
We will identify the sheaves $\pi^*\sV_i$ as better-known sheaves over $\PP^2$ and $S$. To identify the line bundles, we use the stratification of $\Hilb^2$ to reduce to the 2-dimensional situations, as described in Section 4 of \cite{BK04}. The stratification of $\Hilb^2(\hbox{T}^*\PP^1)$ is as the following. One open stratum, two divisors corresponding to the two classes of symplectic reflections, and one codimension 2 stratum which is the punctual Hilbert scheme $X$. We take the 2-dimmensional complimentary $W_1$ to the fixed subspace of $s_1$, and the 2-dimensional complimentary $W_2$ to the fixed subspace of $\sigma$. The restriction of the line bundles $\sV_i$ for $i=1,2,3$ to $W_1$ and $W_2$, we get that $\sV_1=\sO(D_2)$, $\sV_2=\sO(D_1)$, and $\sV_3=\sO(D_1+D_2)$ where $D_1$ is the exceptional divisor coming from the class $s_i$ and $D_2$ from $\sigma$. To identify the rank 2 vector bundle, we use the quiver picture, as will be done in the next subsection.
\end{example}

\subsection{Chern characters via quivers}
To calculate the Chern characters, we want to identify the vector bundle $\sV_\alpha$ as better-known vector bundles. For this, we look at the quiver variety that gives the central fiber of the resolution. According to \cite{Kuz01}, the resolution is the Nakajima variety associated to the quiver affine Dynkin $\widehat{A_{l-1}}$,
$$
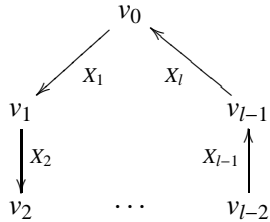
\begin{xymatrix}
{
&v_0\ar[dl]^{X_1}&\\
v_1\ar[d]^{X_2}&&v_{l-1}\ar[ul]^{X_l}\\
v_2&\cdots&v_{l-2}\ar[u]^{X_{l-1}}
}
\end{xymatrix}
$$
with dimension vectors $v_j=2$, $w_0=1$, and $w_j=0$ for $j\neq 0$ and stability $(-1,\cdots,-1)$ (hence being stable means no invariant subrepresentations containing the image of $J:W_0\to V_0$). The McKay correspondence is fixed if we ask the sub-bundle of the tautological bundle generated by the image of $J$ to be the trivial representation.

There is a $\Gm$-action on the quiver variety. On the quiver representation level, this action is given by sending $(X,Y,I,J)$ to $(tX,t^{-1}Y,I,J)$.
Clearly the tautological bundles on the quiver variety are equivariant under this group action. And there are only finitely many fixed points. Therefore, we can calculate the second Chern classes of the tautological bundles using Atiyah-Bott-Berline-Vergne localization.

First let us look at the case when $n=2$:
$$\begin{xymatrix}{
v_0\ar@/^/[r]^{X_1}&v_1\ar@/^/[l]^{X_2}
}
\end{xymatrix}$$
with dimension vector $\dim V=(2,2)$ and $\dim W=(1,0)$. The condition $\mu_\CC(X,Y,I,J)=0$ means $X_1Y_1=Y_2X_2$ and $Y_1X_1=X_2Y_2$.
\par
The central fiber is the moduli space of the nilpotent representations of this quiver satisfying these conditions. The rank-2 vector bundle is the summand of the tautological bundle corresponding to $V_2$. There are two possibilities to get such a representation.
\par
Case 1: $\ker X_1=\ker Y_2$ both 1-dim, and do not contain $\im J$ in $V_0$. And $X_1(J)$ and $Y_2(J)$ are linearly independent in $V_1$. (Here and in what follows we don't distinguish between $J$ and $J(1)$.) We can take the basis for $V$ as follows. Take $J=(0,1)$ and any non-zero vector in $\ker X_1=\ker Y_2$ to be $(1,0)$. Take $X_1(J)$ to be $(1,0)$ and $Y_2(J)$ to be $(0,1)$. Thus, under this basis, $X_1=\left(\begin{smallmatrix}0&1\\0&0\end{smallmatrix}\right)$, $Y_2=\left(\begin{smallmatrix}0&0\\0&1\end{smallmatrix}\right)$, $Y_1=\left(\begin{smallmatrix}a&b\\0&0\end{smallmatrix}\right)$, and $X_2=\left(\begin{smallmatrix}c&a\\0&0\end{smallmatrix}\right)$. The homogeneous coordinates $[a,b,c]$ dose not depend on the choice of basis for $\ker Y_2$. Note that here we do not allow $a^2=bc$. Clearly the restriction of $\sV_4$ to this part is trivial.
\par
Case 2: $X_1(J)$ and $Y_2(J)$ are linearly independent in $V_2$, and $Y_1|_{\im X_1}=0$ $X_2|_{\im Y_2}=0$ $Y_2|_{\im X_2}=0$ and $X_1|_{\im Y_1}=0$. We take the basis for $V_1$ as before, $J=(1,0)\in V_0$, and $(0,1)\in V_0$ arbitrary. Under this basis, $X_1=\left(\begin{smallmatrix}1&x\\0&0\end{smallmatrix}\right)$, $Y_2=\left(\begin{smallmatrix}0&0\\1&w\end{smallmatrix}\right)$, $Y_1=\left(\begin{smallmatrix}0&a\\0&b\end{smallmatrix}\right)$, and $X_2=\left(\begin{smallmatrix}c&0\\d&0\end{smallmatrix}\right)$. These coordinates satisfies the relations $a+bx=0$ and $c+dw=0$. The homogeneous coordinates $[b,d]$ does not depend on the choice of basis hence form a $\PP^1$. All such representations form the total space of the bundle $\sO(1)$ over $\PP^1$. Clearly the restriction of $\sV_4$ to this part is also trivial.
\par
There is a copy of $\PP^1$ sitting as the boundaries of both quasi-projective varieties above. It correspond to the common devisor $\PP^1$ in $\PP^2$ and $S$. This type of representations have $X_1$ proportional to $Y_2$ and $\rank X_2$, $\rank Y_1\leq 1$. Hence $X_2$ is proportional to $Y_1$. The restriction of $\sV_4$ to this part is $\sO(1)\oplus\sO(3)$.
\par

The weights of the tautological bundle and the tangent bundle at the fixed points on the irreducible component $\PP^2$ are summarized in the following table.

\begin{tabular}{|c|c|c|c|}

  \hline

  Fixed point &$[1,0,0]$ & $[0,1,0]$ & $[0,0,1]$ \\
  \hline
  $e^{\Gm}(\hbox{T}\PP^2)$ & $8u^2$ & $8u^2$ & $-4u^2$ \\
  $c_2^{\Gm}(\sV)$&$3u^2$ & $3u^2$ & $-u^2$ \\
  \hline
\end{tabular}

The weights of the tautological bundle and the tangent bundle at the fixed points on the irreducible component $S=\PP(\hbox{T}\PP^1\oplus \hbox{T}^*\PP^1)$ are summarized in the following table.

\begin{tabular}{|c|c|c|c|c|}

  \hline

  Fixed point &$[1,0]$ on $\PP(\hbox{T}\PP^1)$ & $[1,0]$ on $\PP(\hbox{T}^*\PP^1)$  & $[0,1]$ on $\PP(\hbox{T}\PP^1)$ &$[0,1]$ on $\PP(\hbox{T}\PP^1)$  \\
  \hline
  $e^{\Gm}(\hbox{T}S)$ & $-8u^2$ & $8u^2$ & $-8u^2$ & $8u^2$ \\
  $c_2^{\Gm}(\sV)$&$3u^2$ & $-u^2$&$3u^2$ & $-u^2$ \\
  \hline
\end{tabular}

Now we can calculate that $c_2(\sV_4)=s+p$ and $c_1(\sV_4)=2q+4f+c_0$ ($c_1(\sV_4)$ can also be obtained by using the stratification and passing to the dimension 2 case).

Summarizing all the discussions above in this section we get.
\begin{lemma}
The chern characters of $\sV_i$'s are as follows:
\begin{eqnarray*}
\ch(\sV_0)&=&1;\\
\ch(\sV_1)&=&1+c_0+2f+q+p/2;\\
\ch(\sV_2)&=&1+2f+q+p/2;\\
\ch(\sV_3)&=&1+c_0+4f+2q+2p+2s;\\
\ch(\sV_4)&=&2+2q+4f+c_0+p+s.
\end{eqnarray*}
\end{lemma}

\begin{remark}
In the calculation we use the stability condition $(-1,\dots,-1)$. The Hilbert scheme of point correspondes to Nakajima quiver variety but with a different stability condition. It is not hard to see, by analysing the weights to the tangent spaces, that the fixed points on the Hilbert scheme matche up with the above table.
\end{remark}

The natural embedding of $Hilb^2_0(\hbox{T}^*\PP^1)$ into $\Hilb^2_C(\widetilde{\Aff^2/\ZZ_l})$  as $S_i\coprod\PP^2_i$ can be constructed on the quiver level. For any representation $\xi$ of $\widehat{A_1}$ that lie in $Hilb^2_0(\hbox{T}^*\PP^1)$, we associate a presentation $\eta$ of $\widehat{A_{l-1}}$ as follows. Take $X_1,\cdots,X_{i-1}$ and $X_{i+2},\cdots,X_{l}$ all to be the identity. Identify $V_0,\cdots,V_{i-1},V_{i+1},\cdots,V_{l-1}$ using those $X$ maps which we take to be the identity. Let the $X_i$ map of $\eta$ to be the $X_1$ map of $\xi$, using the identification of $V_0$ and $V_{i-1}$ as above. Similarly, let the $Y_i$ map of $\eta$ to be the $Y_1$ map of $\xi$; the $X_{i+1}$ map of $\eta$ to be the $X_2$ map of $\xi$; the $Y_{i+1}$ map of $\eta$ to be the $Y_2$ map of $\xi$.  Then, the ADHM equation uniquely determines $Y_1,\cdots,Y_{i-1}$ and $Y_{i+2},\cdots,Y_{l}$. Note that this embedding is not equivariant under the $\Gm$-action.

Now we look at the $n=3$ case.

We express the open part of the $P_{1,2}$ component, when $X_1$ and $X_3$ has rank 1, in terms of quiver varieties. In this case, the representations have $\ker X_1=\im Y_1=\ker Y_3=\im X_3$, $\ker X_2=\im Y_2$, and $\im X_2=\ker Y_2$. We choose a basis for $V_0$ to be $J(1)$ and an arbitrary non-zero vector in $\im Y_1$. We take the basis for $V_1$ to be $X_1J(1)$ and $Y_2Y_3J(1)$, and the basis for $V_2$ to be $Y_3J(1)$ and $X_2X_1J(1)$. Let $Y_1=\left(\begin{smallmatrix}0&0\\a&b\end{smallmatrix}\right)$, $X_3=\left(\begin{smallmatrix}0&0\\c&a\end{smallmatrix}\right)$.
Under this choice of basis, the representations are determined by the homogenous coordinates $[a,b,c]$. An open part of this $\PP^2$ is the open set of the component $P_{1,2}$ we are looking for.
There is one fixed point in this open set which corresponds to $([1:0],[0:1])\in \PP^1\times\PP^1$. The other fixed points all lie in the intersection of $P_{1,2}$ with $\PP^2_1$ or $\PP^2_2$, and are not in this open set.

\noindent\begin{tabular}{|c|c|c|c|c|c|}

  \hline

  Fixed point &$[0,1,0]$ on $\PP^2_1$ & $[0,0,1]$ on $\PP^2_1$  & $[1,0,0]$ on $\PP^2_2$ &$[0,0,1]$ on $\PP^2_2$& $([1,0],[0,1])$ on $\PP^1\times\PP^1$ \\
  \hline
  $e^{\Gm}(\hbox{T}P_{1,2})$ & $-8u^2$ & $4u^2$ & $-8u^2$ & $4u^2$ & $-4u^2$\\
  $c_2^{\Gm}(\sV)$&$-2u^2$ & $-2u^2$&$-2u^2$ & $-2u^2$ & $-2u^2$\\
  \hline
\end{tabular}

For $l\geq3$, the natural embedding of $\Hilb^3_C(\widetilde{\Aff^2/\ZZ_3})$ into $\Hilb^2_C(\widetilde{\Aff^2/\ZZ_l})$ as $S_i\coprod S_{i+1}\coprod\PP^2_i\coprod\PP^2_{i+1}\coprod\PP_{i,i+1}$ gives the second Chern classes of the bundle $\sV_{0,k}$ on any of the components $ P_{i,i+1}$.

\begin{lemma}The second Chern class of the rank 2 vector bundle $\sV_{0,i}$ is:
\[c_2(\sV_{0,i})=s_i+p_i.\]
\end{lemma}

\begin{proof}[Proof of Proposition~\ref{prop: chern char}]
By Lemma~\ref{lem: k-grp eq} we have
\begin{eqnarray*}
\ch(\sV_{i,j})&=&\ch(\sV_i)\cdot(\ch(\sV_0)+\ch(\sV_\sigma)+\ch(\sV_{0,i})-\ch(\sV_{0,j}))\\
&=&(1+d_i+p_i/2)\cdot(2+d_0+d_j-d_i+p_j/2-p_i/2+s_j-s_i+1/2(\sum p_i-\sum p_{j,j+1}))\\
&=&2+d_j+d_i+d_0+\sum p_k/2-\sum p_{k,k+1}/2+p_{i,j}+p_i/2+p_j/2+s_i+s_j.
\end{eqnarray*}
\end{proof}

\section{The central charge}\label{sec: central_charge}

As has been mentioned above, the central charge can be obtained from the knowledge of the Chern character map. In this section, we illutrate this in the special case $n=2$ and $\Gamma_1=\ZZ/l$.
\subsection{The $B_2$-case}\label{subsec: central_charge_B2}

By calculation it is easy to know that $$\ch(\sV_4)=1/2\ch(\sV_1)+3/2\ch(\sV_2)+1/2\ch(\sV_1\otimes\sV_2)-1/2\ch(\sV_2^2).$$
\par
For any coherent sheaf $\mathscr{F}$ on $X$, we consider the linear functional on $H^*(\Hilb^2)$ defined as $\mathfrak{l}_\mathscr{F}(\ch(\mathscr{M}))=\chi(\mathscr{F}\otimes\pi^*\mathscr{M})$ for $\mathscr{M}\in Coh(\Hilb^2)$. We now calculate the polynomials $$\mathfrak{l}_{\sL_j}(a,b):=\mathfrak{l}_{\sL_j}(\ch(\sO(aD_1+bD_2))).$$ This polynomial can be written as $$\mathfrak{l}_\mathscr{F}(a,b)=\mathfrak{l}_\mathscr{F}(1)+a \mathfrak{l}_\mathscr{F}(d_1)+b \mathfrak{l}_\mathscr{F}(d_2)+(a^2+b^2+2ab)\mathfrak{l}_\mathscr{F}(p/2)+4ab\mathfrak{l}_\mathscr{F}(s/2).$$ We denote the coefficients by $C_{\mathscr{F}}^0=\mathfrak{l}_\mathscr{F}(1)$, $C_{\mathscr{F}}^1=\mathfrak{l}_\mathscr{F}(d_1)$, $C_{\mathscr{F}}^2=\mathfrak{l}_\mathscr{F}(d_2)$, $C_{\mathscr{F}}^3=\mathfrak{l}_\mathscr{F}(s)$, and $C_{\mathscr{F}}^4=\mathfrak{l}_\mathscr{F}(p)$.
\par
We know that $\chi(\sHom(\sV_i,\sL_j))=\delta_{i,j}$ for $i,j=0,\cdots,4$. This tells us the values of the polynomials $\mathfrak{l}_{\sL_j}$ at the points $(0,0)$, $(-1,0)$, $(0,-1)$, $(-1,-1)$ and $(-2,0)$. Now we plug-in and solve these systems of linear equations to get

$$
  \begin{array}{ccccc}
    C_0^0=1 & C_0^1=3/2 & C_0^2=3/2 & C_0^3=1/2 & C_0^4=0 \\
    C_1^0=0 & C_1^1=1/2 & C_1^2=-1/2 & C_1^3=1/2 & C_1^4=-1/2 \\
    C_2^0=0 & C_2^1=-1/2 & C_2^2=1/2 & C_2^3=1/2 & C_2^4=-1/2 \\
    C_3^0=0 & C_3^1=1/2 & C_3^2=1/2 & C_3^3=1/2 & C_3^4=0 \\
    C_4^0=0 & C_4^1=-1 & C_4^2=-1 & C_4^3=-1 & C_4^4=1/2 \\
  \end{array}
$$
\par
Plug-in these coefficients and we get
\begin{eqnarray*}
\mathfrak{l}_{\sL_0}&=&1/2(a+b+1)(a+b+2);\\
\mathfrak{l}_{\sL_1}&=&1/2(a-b+1)(a-b);\\
\mathfrak{l}_{\sL_2}&=&1/2(a-b-1)(a-b);\\
\mathfrak{l}_{\sL_3}&=&1/2(a+b+1)(a+b);\\
\mathfrak{l}_{\sL_4}&=&-a-b-a^2-b^2.
\end{eqnarray*}
Now we define the dimension polynomials $$P_{\sL_j}(a,b)=\chi(\mathscr{E}_0\otimes\sL_j\otimes\sO(a,b)),$$ where $\mathscr{E}_0$ is the self-dual splitting bundle on $\Hilb^2$ in \cite{BFG06}. (Their reparametrizations $P_{\sL_j}(\alpha,\beta)$ with $\alpha=ap$ and $\beta=bp$ gives the dimension of the irreducible objects.)
They can be computed as $$\chi(\mathscr{E}_0\otimes\sL_j\otimes\sO(a,b))=\sum_i[V_i:(k[x]/(x^p))^{\otimes2}]\chi(\sV_i^*\otimes\sL_j\otimes\sO(a,b)),$$ where $V_i$'s are the irreducible representations of $B_2$.
We plug-in the equality in $K_{\mathbb{Q}}$ that $$\sV_4=1/2\sV_1+3/2\sV_2+1/2\sV_3-1/2\sO(2,0),$$ and get $P_{\sL_j}(a,b)={{p^2-1}\over{8}}\mathfrak{l}_{\sL_j}(a,b)+{{p^2-2p-3}\over{4}}\mathfrak{l}_{\sL_j}(a,b-1)+{{p^2-p}\over{2}}\mathfrak{l}_{\sL_j}(a-1,b)+{{p^2-4p+3}\over{4}}\mathfrak{l}_{\sL_j}(a-1,b-1)-{{p^2-1}\over{8}}\mathfrak{l}_{\sL_j}(a-2,b)$.

Then we define the central charge polynomials as $Z_j(a,b)=\lim_{p\to\infty}1/p^2P_{\sL_j}(a,b)$. An easy calculation shows
\begin{eqnarray*}
Z_0&=&1/8(2a+2b+1)^2;\\
Z_1&=&1/8(-2a+2b-1)^2;\\
Z_2&=&1/8(-2a+2b+1)^2;\\
Z_3&=&1/8(2a+2b-1)^2;\\
Z_4&=&-1/4(4a^2+4b^2-1).
\end{eqnarray*}

One can see that $Z_4$ is an irreducible polynomial and its real zero locus is a circle, and it takes positive values in the region bounded by the zero locus of $Z_0,\cdots, Z_3$.
\par

\subsection{The $\Gamma_1=\ZZ/l$-case}
We solve for the geometric basis of $H^*(\Hilb^2_{\ZZ/l},\QQ)$ in terms of $\ch(\sV_\alpha)$.
\begin{eqnarray*}
s_i&=&\ch(\sV_{0,i})-\ch(\sV_\sigma)-\ch(\sV_i);\\
p_i&=&\ch(\sV_0)+\ch(\sV_{\sigma,i})+\ch(\sV_\sigma)+\ch(\sV_i)-2\ch(\sV_{0,i});\\
p_{i,j}&=&\ch(\sV_\sigma)+\ch(\sV_0)+\ch(\sV_{i,j})-\ch(\sV_{0,i})-\ch(\sV_{0,j});\\
d_i&=&\ch(\sV_i)/2+\ch(\sV_{0,i})-\ch(\sV_{\sigma,i})/2-3/2\ch(\sV_{0})-\ch(\sV_\sigma)/2;\\
d_0&=&-3/2\ch(\sV_0)+1/2\ch(\sV_\sigma)+1/2\ch(\sV_{0,1})+1/2\ch(\sV_{0,l-1})\\
&&-\sum\ch(\sV_{i})/2-\sum\ch(\sV_{\sigma,i})/2+1/2\sum_{j=1}^{l-2}\ch(\sV_{j,j+1});\\
1&=&\ch(\sV_0).
\end{eqnarray*}

We have \[\ch(\sO(n_0D_0+\sum n_iD_i))\]\[=1+n_0d_0+\sum_in_id_i+n_0^2/2(\sum p_i-\sum p_{i,i+1})+\sum n_i^2/2p_i+\sum_{i> j}n_in_jp_{i,j}+\sum_in_0n_i(p_i+2s_i).\]
Hence, the polynomials $\mathfrak{l}_{\sL_\alpha}$ can be calculated as below.
\begin{eqnarray*}
\mathfrak{l}_{\sL_0}&=&1/2(n_0+\sum n_i-1)(n_0+\sum n_i-2);\\
\mathfrak{l}_{\sL_\sigma}&=&1/2(n_0-\sum n_i+1)(n_0-\sum n_i);\\
\mathfrak{l}_{\sL_i}&=&1/2(n_i-n_0)(n_i-n_0+1);\\
\mathfrak{l}_{\sL_{\sigma,i}}&=&1/2(n_0+n_i)(n_0+n_i-1);\\
\mathfrak{l}_{\sL_{0,i}}&=&1/2n_0(1-n_0)+n_i(1-\sum_{j=1}^{l-1} n_j)\hbox{ if $i=1$ or $l-1$};\\
\mathfrak{l}_{\sL_{0,i}}&=&n_i(1-\sum_{j=1}^{l-1} n_j)\hbox{ if $i\neq1$ or $l-1$};\\
\mathfrak{l}_{\sL_{j,i}}&=&1/2n_0(1-n_0)+n_in_j\hbox{ if $i-j=\pm 1$};\\
\mathfrak{l}_{\sL_{j,i}}&=&n_in_j\hbox{ otherwise}.
\end{eqnarray*}

Also, there is a basis given by Chern character of line bundles, which can be chosen as $\ch\sV_i$, $\ch\sV_{\sigma,i}$, $\ch\sV_{i}^2$, and $\ch(\sV_i)\cdot\ch(\sV_j)$.
As can be easily checked,
\[\ch(\sV_i^2)=1+2d_i+2p_i,\]
\[\ch(\sV_i)\cdot\ch(\sV_j)=1+d_i+d_j+1/2(p_i+p_j)+p_{i,j}.\]
In the group $K_{\QQ}$, we have the change of basis
\[\sV_{0,i}=3/2\sV_i+1/2\sV_\sigma+1/2\sV_{\sigma,i}-1/2\sV_i^2,\]
\[\sV_{i,j}=\sV_i\cdot\sV_j+1/2\sV_{\sigma,i}-1/2\sV_i^2+1/2\sV_i+1/2\sV_{\sigma,j}-1/2\sV_j^2+1/2\sV_j.\]

We will decompose $(k[x]/x^p)^2$ into isotypical components. But in the decomposition, we only care about the behavior for $p$ large enough. Therefore, we will keep only highest order terms with respect to $p$ in the multiplicities of the irreducible representations. the multiplicity of irreducibles is  $[(k[x]/x^p)^2:V_{i,j}]=(p/l)^2$ for $i\neq j$, and $[(k[x]/x^p)^2:V_i]=[(k[x]/x^p)^2:V_{\sigma,i}]=1/2(p/l)^2$.

The dimensional polynomials are, forgetting terms involving $p$'s power less than or equal to 2, \begin{eqnarray*}P_{\sL_{\alpha}}(n_0,n_1,\dots,n_{l-1})&=&(p/l)^2(1/2\mathfrak{l}_\alpha(n_0,n_1,\dots,n_{l-1})+l/2\mathfrak{l}_\alpha(n_0-1,n_1,\dots,n_{l-1})\\
&+&\sum_{i=1}^{l-1}(l+2)/2\mathfrak{l}_\alpha(n_0,n_1,\dots,n_i-1,\dots,n_{l-1})\\
&+&\sum_{i=1}^{l-1}l/2\mathfrak{l}_\alpha(n_0-1,n_1,\dots,n_i-1\dots,n_{l-1})\\
&+&\sum_{i=1}^{l-1}-(l-1)/2\mathfrak{l}_\alpha(n_0,n_1,\dots,n_i-2,\dots,n_{l-1})\\
&+&\sum_{i>j}\mathfrak{l}_\alpha(n_0,n_1,\dots,n_j-1,\dots,n_i-1,\dots,n_{l-1})).
\end{eqnarray*}

So the central charge polynomials are
\begin{eqnarray*}
Z_0&=&1/2((n_0+\sum n_i)-(3-1/l))^2;\\
Z_\sigma&=&1/2((n_0-\sum n_i)+(1-1/l))^2;\\
Z_i&=&1/2((n_i-n_0)+(1-1/l))^2;\\
Z_{\sigma,i}&=&1/2((n_0+n_i)-(1+1/l))^2\\
Z_{0,i}&=&-l/2(n_0-1)^2-(n_i-1/l)(\sum n_j-(2-1/l))\hbox{ if $i=1$ or $l-1$};\\
Z_{0,i}&=&-(n_i-1/l)(\sum n_j-(2-1/l))\hbox{ if $i\neq1$ or $l-1$};\\
Z_{j,i}&=&-l/2(n_0-1)^2+(n_i-1/l)(n_j-1/l)\hbox{ if $i-j=\pm 1$};\\
Z_{j,i}&=&(n_i-1/l)(n_j-1/l)\hbox{ otherwise}.
\end{eqnarray*}

\section{The $t$-structures associated to alcoves}\label{sec: t-struct_alco}
In this section we study the $t$-structures associated to the alcoves in the affine hyperplane arrangement defined in Section~\ref{sec: quant_char_p}. 

In the alcove containing the origin, we associate the $t$-structure  coming from the derived equivalence $$D^b(\Coh\Hilb^n)\cong D^b(\Coh_{\Gamma_n}\Aff^{2n}).$$
We need to find out the $t$-structures associated to other alcoves. We start with the case when $\Gamma_n\cong B_2$.

\subsection{The $B_2$-case}
We need to calculate the Ext's among the simple objects in $Coh_{B_2}\hbox{-}\CC^4$. The result is summarized as follows: $$\Ext^\bullet(\sL_i,\sL_j)=$$
\begin{tabular}{|c|c|c|c|}

  \hline
  & \hbox{$\sL_i=\sL_j$ both rank 1} & \hbox{$\sL_i\neq\sL_j$ both rank 1}  &  \hbox{$\sL_i\neq\sL_j$ one of them has rank 2} \\
  \hline
  $\deg=0$&$\CC$ & 0& 0 \\
  $\deg=1$&0 & 0 & $\CC^2$ \\
  $\deg=2$&$\CC$ & $\CC$ or $\CC^3$& 0 \\
  $\deg=3$&0 & 0 & $\CC^2$ \\
  $\deg=4$&$\CC$ & 0 & 0 \\
  \hline
\end{tabular}

In the case that both $\sL_i$ and $\sL_j$  have rank 1,  $\Ext^2(\sL_i,\sL_j)$ is $\CC^3$ only in the following cases:  One of $i$, $j$ is 0 and the other is 3, or one of them is 1 and the other is 2.
The $\Ext^\bullet(\sL_4,\sL_4)$ is $\CC$ in degree 0 and 4, $\CC^6$ in degree 2, and zero otherwise.

The bilinear pairing $\langle [A],[B]\rangle:=\sum_i(-1)^i\Ext^i(A,B)$ can be expressed under the basis $\{\sL_j\}$ as $$\left(\begin{matrix}
3&1&1&3&-4\\
1&3&3&1&-4\\
1&3&3&1&-4\\
3&1&1&3&-4\\
-4&-4&-4&-4&8
\end{matrix}\right)$$

Let us look at what happens when cross the wall defined by $Z_0=0$. The new abelian category has the same Grothendieck group. The simple objects have classes $[\sL_0]$, $[\sL_1]-[\sL_0]$, $[\sL_2]-[\sL_0]$, $[\sL_3]-3[\sL_0]$, and  $[\sL_4]+2[\sL_0]$.
One can calculate the dual basis to find the classes of their projective covers in the Grothendieck group. They are $[\sV_0]+[\sV_1]+[\sV_2]+3[\sV_3]-2[\sV_4]$, $[\sV_1]$, $[\sV_2]$, $[\sV_3]$, and $[\sV_4]$ respectively.
One can also see, through the bilinear pairing under the new basis, that the new abelian category is not Morita equivalent to the original one, as the bilinear form $$\left(\begin{matrix}
3&-2&-2&-6&2\\
-2&4&4&4&-4\\
-2&4&4&4&-4\\
-6&4&4&12&-4\\
2&-4&-4&-4&4

\end{matrix}\right)$$ does not differ from the original one by a permutation.

The central charge polynomials corresponding to the simple objects in this abelian heart can be obtained from the original ones. More explicitly, they are, respectively,
\begin{eqnarray*}
Z_0&=&1/8(2a+2b+1)^2;\\
Z_1-Z_0&=&-(2a+1)b;\\
Z_2-Z_0&=&-a(2b+1);\\
Z_3-3Z_0&=&1/8((2a+2b-1)^2-3(2a+2b+1));\\
Z_4+2Z_0&=&1/2(2a+1)(2b+1).
\end{eqnarray*}
One can see that $Z_3-3Z_0$ takes positive values in the region bounded by the other 4 polynomials.

In fact, we can do iterated (right) tilting with respect to the simple object $\sL_0$. The intermediate $t$-structure $R_{\sL_0}\catA$ has simple objects $\sL_0[1]$, $\sL_j$ for $j=1,2,3$, and $\sL_4^1$ fitting into the short exact sequence $$0\to \Ext^1(\sL_4,\sL_0)^*\otimes\sL_0\to \sL_4^1\to \sL_4\to 0.$$
The classes of their potential projective covers are $-[\sV_0]+2[\sV_4]$, $[\sV_1]$, $[\sV_2]$, $[\sV_3]$, and $[\sV_4]$ respectively. We can try to do the truncated mutation with respect to $\sV_0$. We get $\sV_j^1=\sV_j$ for $j\neq 0$, and $\sV_j^0$ being the cokernel of $\sV_0\to\Hom(\sV_0,\sV_4)^*_1\otimes\sV_4$. The injectivity of the map $\sV_0\to\Hom(\sV_0,\sV_4)^*_1\otimes\sV_4$ can be checked generically on $\CC^4$. Proposition~\ref{prop: exist_trunc_mutat} tells us that the abelian heart $R_{\sL_0}\catA$ is derived  equivalent to the original category.

If we do tilting again, we get $R_{\sL_0[1]}R_{\sL_0}\catA$ whose simple objects are $\sL_0[2]$, $\sL_j^2$ fitting into the short exact sequence $$0\to \Ext^2(\sL_j,\sL_0)\otimes\sL_0[1]\to\sL_j^2\to\sL_j\to0$$ for $j=1,2,3$, and $\sL_4^2=\sL_4^1$. This is exactly the new abelian category we obtained cross the wall $Z_0=0$.

Similarly, one can start from the initial region and go across the other walls.
As an example, let us look at the $t$-structure associated to the region across the wall defined by $Z_1=0$. The simple objects have classes in the Grothendieck group $[\sL_0]-[\sL_1]$, $[\sL_1]$, $[\sL_2]-3[\sL_1]$, $[\sL_3]-[\sL_1]$, and $[\sL_4]+2[\sL_1]$. Their corresponding projective covers have classes $[\sV_0]$, $[\sV_1]+[\sV_0]+3[\sV_2]+[\sV_3]-2[\sV_4]$, $[\sV_2]$, $[\sV_3]$, and $[\sV_4]$.

We can do iterated tilting to find out the complexes in the original abelian category represents these simple objects. The simple object $\sL_1^2$ corresponding to $\sL_1$ is $\sL_1[2]$. For $i=0,2,3$, the simple object $\sL_i^2$ fits into short exact sequences $$0\to \Ext^2(\sL_i,\sL_1)^*\otimes\sL_1[1]\to \sL_i^2\to \sL_i\to0.$$
And $\sL_4^2$ fits into the short exact sequence $$0\to\Ext^1(\sL_4,\sL_1)^*\otimes\sL_1\to\sL_4^2\to\sL_4\to0.$$

Now we look at the $t$-structure associated to the region across the wall defined by $Z_3=0$. The simple objects have classes in the Grothendieck group $[\sL_0]-3[\sL_3]$, $[\sL_1]-[\sL_3]$, $[\sL_2]-[\sL_3]$, $[\sL_3]$, and $[\sL_4]+2[\sL_3]$. Their corresponding projective covers have classes $[\sV_0]$, $[\sV_1]$, $[\sV_2]$, $[\sV_3]+3[\sV_0]+[\sV_1]+[\sV_2]-2[\sV_3]$, and $[\sV_4]$.

If we do iterated tilting, we can find that the simple object $\sL_3^2$ corresponding to $\sL_3$ is $\sL_3[2]$. For $i=0,1,2$, the simple object $\sL_i^2$ fits into short exact sequences $$0\to \Ext^2(\sL_i,\sL_3)^*\otimes\sL_3[1]\to \sL_i^2\to \sL_i\to0.$$
And $\sL_4^2$ fits into the short exact sequence $$0\to\Ext^1(\sL_4,\sL_3)^*\otimes\sL_3\to\sL_4^2\to\sL_4\to0.$$

Now it is turn to look at the $t$-structure associated to the region across the wall defined by $Z_2=0$. The simple objects have classes in the Grothendieck group $[\sL_0]-[\sL_2]$, $[\sL_1]-3[\sL_2]$, $[\sL_2]$, $[\sL_3]-[\sL_2]$, and $[\sL_4]+2[\sL_2]$. Their corresponding projective covers have classes $[\sV_0]$, $[\sV_1]$, $[\sV_2]+[\sV_0]+3[\sV_2]+[\sV_3]-2[\sV_4]$, $[\sV_3]$, and $[\sV_4]$.

Similarly we can find that the simple object $\sL_2^2$ corresponding to $\sL_2$ is $\sL_2[2]$. For $i=0,1,3$, the simple object $\sL_i^2$ fits into short exact sequences $$0\to \Ext^2(\sL_i,\sL_2)^*\otimes\sL_2[1]\to \sL_i^2\to \sL_i\to0.$$
And $\sL_4^2$ fits into the short exact sequence $$0\to\Ext^1(\sL_4,\sL_2)^*\otimes\sL_2\to\sL_4^2\to\sL_4\to0.$$

Note that the region bounded by the walls $a=\pm1/2$ and $b=\pm1/2$ is a fundamental domain of the $H^2(\Hilb^2(\hbox{T}^*\PP^1),\ZZ)\cong\ZZ^2$ action on $H^2(\Hilb^2(\hbox{T}^*\PP^1),\RR)$. Denote this domain by $D_0$.
Summarizing the discussion in this subsection, we get the following proposition.

\begin{prop}
There is a real variation of stability conditions on $D_0$ (hence any translation of it by $H^2(\Hilb^2(\hbox{T}^*\PP^1),\ZZ)$) whose $t$-structure at the origin is $Coh_{B_2}(\CC^4)$ with central charge polynomials given by $Z_i$ defined in Subsection~\ref{subsec: central_charge_B2}.
\end{prop}

To summarize the description of the hyperplane arrangement and the $t$-structures associated to alcoves in this case, the following picture is part of the hyperplane arrangement.

The $t$-structure associated to the alcove labeled by $t_0$ is the so called orbifold (or BKR) $t$-structure, whose tilting generator can be chosen to be the Haiman\rq{}s Procesi bundle. The four $t$-structure associated to alcoves adjacent to $t_0$ are obtained from $t_0$ by $\PP^2$-semi-reflections. The tilting generators for these $t$-structures are obtained from the truncated mutations described in Section~\ref{sec: mut}.

The rest of the alcoves can be obtained from them by a shifting. For example, the alcove $t_2$ and $t_2\rq{}$ differ by a translation, hence the corresponding $t$-structures differs by twisting by a line bundle.

\begin{center}
\begin{tikzpicture}[scale=2.5]
 \coordinate (y) at (0,1.7);
 \coordinate (x) at (0.7,0);
 \draw[axis] (y)--(0,0)--(x);
 \draw[thick] (0,-0.5)--(0.5,0);
 \draw[thick] (0,-0.5)--(-0.5,0);
 \draw[thick] (0.5, 1)--(-0.5,0);
 \draw[thick] (0.5, 0)--(-0.5,1);
 \draw[thick] (-0.5, 0.5)--(0.5,0.5);
 \draw[thick] (0.5, -0.5)--(0.5,1);
 \draw[thick] (-0.5, -0.5)--(-0.5,1);
 \draw[thick] (-0.5, 1)--(0,1.5);
 \draw[thick] (0.5, 1)--(0,1.5);
 \draw[thick] (-0.5,-0.5)--(0.5,-0.5);
 \draw (0.35,0.35) node {$t_1$};
  \draw (0.35,-0.35) node {$t_2\rq{}$};
 \draw (0,0) node[anchor=east] {$t_0$};
 \draw (0.35,0.65) node {$t_2$};
 \draw (0,1) node[anchor=east] {$t_3$};
 \draw (y) node[anchor=east] {$b$};
 \draw (x) node[anchor=north] {$a$};
 \end{tikzpicture}
\end{center}

\begin{prop}
The functor $F=-\otimes\sV_2: D^b(\Hilb^2(\hbox{T}^*\PP^1))\to D^b(\Hilb^2(\hbox{T}^*\PP^1))$ with the source endowed with the $t$-structure coming from $Coh_{B_2}(\CC^4)$ is a perverse equivalence for suitable filtration and perversity function.
\end{prop}
\begin{proof}
Define a filtration on $\catA=Coh_{B_2}(\CC^4)$ as follows: $\catA_1$ is the Serre subcategory generated by the simple objects $\sL_0$ and $\sL_1$; $\catA_2$ is generated by $\sL_{0,1}$ in addition to $\catA_1$; and $\catA_3=\catA$. Define the perversity function to be $p(1)=2$, $p(2)=1$, $p(3)=0$.

Using the notations as in Diagram~\ref{eqn:diag_res}, the bundle $\sV^2$ comes from $(\hbox{T}^*\PP^1)^2$. Using projection formula, the complexes $F(\sV_i)$ can be  calculated in $Coh_{\mathfrak{S}_2}((\hbox{T}^*\PP^1)^2)$. Let $\sO_r$ be the trivial line bundle on $(\hbox{T}^*\PP^1)^2$ endowed with the reflection representation. We have
\[F(\sV_0)=\sL\boxtimes\sL\cong\sV_1;\]
\[F(\sV_1)=(\sL^2)\boxtimes(\sL^2)\cong(\sO\to2\sL)\boxtimes(\sO\to2\sL)\cong(\sO\to2\sV_{0,1}\to4\sV_1);\]
\[F(\sV_2)=\sO_r\otimes(\sL\boxtimes\sL)\cong\sV_3;\]
\[F(\sV_3)=(\sL\boxtimes\sL)\otimes(\sO_r\otimes(\sL\boxtimes\sL))=(\sV_2\to2\sV_{0,1}\to4\sV_3);\]
\[F(\sV_4)=(\sL^2\boxtimes\sL)\oplus(\sL\boxtimes\sL^2)\cong(\sV_4\to4\sV_1).\]
This proves the statement.
\end{proof}

\subsection{The $\Gamma_1=\ZZ_l$-case}
In the initial alcove which is bounded by the walls
\begin{eqnarray*}
1/2((n_0+\sum n_i)-(3-1/l))^2&=&0\\
1/2((n_0-\sum n_i)+(1-1/l))^2&=&0\\
1/2((n_i-n_0)+(1-1/l))^2&=&0\hbox{ for all $i=1,\cdots,l-1$}\\
1/2((n_0+n_i)-(1+1/l))^2&=&0\hbox{ for all $i=1,\cdots,l-1$},
\end{eqnarray*}
we associate to it the $t$-structure coming from the derived equivalence $D^b(\Hilb^2_{\ZZ/l})\cong D^b(\Coh_{\Gamma_2}(\Aff^{2n}))$. The simple objects are labeled by the irreducible representations.

If we go across the wall defined by $1/2((n_0+\sum n_i)-(3-1/l))^2=0$, the $t$-structure is the one coming from $R_{\sL_0[1]}R_{\sL_0}\Coh_{\Gamma_2}(\Aff^{2n})$. By Section~\ref{subsec: tilting}, we know the heart of this new $t$-structure is also a finite length category. The simple objects are $\sL_\alpha''$ as defined in Section~\ref{subsec: tilting}. The classes of the simple objects in the Grothendieck group are $[\sL_0]$, $[\sL_\sigma]-3[\sL_0]$, $[\sL_i]$, $[\sL_{\sigma,i}]$, $[\sL_{0,1}]+[\sL_0]$, $[\sL_{0,l-1}]+[\sL_0]$, $[\sL_{0,i}]$ for $i\neq 2,l-2$, and $[\sL_{i,j}]$. The new alcove is bounded by the walls
\begin{eqnarray*}
\sum n_j-2+1/l&=&0\\
1/2((n_0+\sum n_i)-(3-1/l))^2&=&0\\
1/2((n_i-n_0)+(1-1/l))^2&=&0\hbox{ for all $i=1,\cdots,l-1$}
\end{eqnarray*}
where on the first wall, both $Z_{\sL_{0,1}''}$ and $Z_{\sL_{0,l-1}''}$ vanish, and both of order 1.

If we go across the wall $\sum n_j-2+1/l=0$, the new $t$-structure is obtained by doing tilting with respect to $\sL_{0,1}''$ and $\sL_{0,l-1}''$. The classes of the simple objects in the Grothendieck group are $[\sL_0]$, $[\sL_\sigma]-3[\sL_0]$, $[\sL_i]$ for $i=2,\cdots,l-2$, $[\sL_1]+[\sL_{0,1}]+[\sL_0]$, $[\sL_{l-1}]+[\sL_{0,l-1}]+[\sL_0]$, $[\sL_{\sigma,i}]$, $-[\sL_{0,1}]-[\sL_0]$, $-[\sL_{0,l-1}]-[\sL_0]$, $[\sL_{0,i}]$ for $i\neq 2,l-2$, and $[\sL_{i,j}]$. The new alcove is bounded by the walls
\begin{eqnarray*}
1/2((n_0-1)+\sum_{j=2}^{l-1}(n_j-1/l)-1)^2&=&0\\
1/2((n_0-1)+\sum_{j=1}^{l-2}(n_j-1/l)-1)^2&=&0\\
1/2((n_0-\sum n_i)+(1-1/l))^2&=&0\\
\sum_{j=1}^{l-1} n_j-2+1/l&=&0\\
1/2((n_i-n_0)+(1-1/l))^2&=&0\hbox{ for all $i=1,\cdots,l-1$}
\end{eqnarray*}

Then similar to the $B_2$-case, symmetry gives the $t$-structures associated to the other alcoves.

\newcommand{\arxiv}[1]
{\texttt{\href{http://arxiv.org/abs/#1}{arXiv:#1}}}
\newcommand{\doi}[1]
{\texttt{\href{http://dx.doi.org/#1}{doi:#1}}}
\renewcommand{\MR}[1]
{\href{http://www.ams.org/mathscinet-getitem?mr=#1}{MR#1}}

\end{document}